\documentclass{amsart}
\usepackage{amsmath,amsthm,amssymb,amsfonts}
\usepackage[latin1]{inputenc}

\usepackage[pdftex]{graphicx}
\DeclareGraphicsExtensions{.png,.pdf,.jpg}
\usepackage{color}
\usepackage{hyperref}

\newtheorem{lema}{Lemma}
\theoremstyle{definition}
\newtheorem{teor}{Theorem}
\newtheorem{append}{Note}
\newtheorem{rem}{Remark}
\newtheorem{prop}{Proposition}

\let\abs=\envert
\theoremstyle{remark}

\usepackage{color}

\begin{document}

\title{Revisiting floating bodies}

\author{\'Oscar Ciaurri, Emilio Fern\'andez, and Luz Roncal}
\email{oscar.ciaurri@unirioja.es, eferna35@gmail.com, luz.roncal@unirioja.es}

\address{Departamento de Matem\'aticas y Computaci\'on,
Universidad de La Rioja, Edificio J.~L.~Vives, Calle Luis de Ulloa
s/n, 26004 Logro\~no, Spain.}

\begin{abstract}
The conic sections, as well as the solids obtained by revolving these curves, and many of their surprising properties,
were already studied by Greek mathematicians since at least the fourth century B.C.
Some of these properties come to the light, or are rediscovered, from time to time. In this paper we
characterize the conic sections as the plane curves whose tangent lines cut off from a certain similar curve segments of constant area. We also characterize some quadrics as the surfaces whose tangent planes cut off from a certain similar surface compact sets of constant volume. Our work is developed in the most general multidimensional case.
\end{abstract}

\maketitle
\thispagestyle{empty}

\section{Introduction}

Throughout literature, we can find several properties of the plane regions obtained by cutting off with a line from conic sections and of the solids obtained by cutting off with a plane from quadrics of revolution. Let us show a brief chronological exposition of some results concerning the topic so far.

In his book \emph{Quadrature of the Parabola} (QP), Archimedes proved the following result (see \cite[(QP), Propositions 17 and 24, p. 246-252]{heath}):

\begin{prop}[QP]
\label{prop:mainQP}
The area of any segment of a parabola is four thirds of the area of the triangle which has the same base as the segment and equal height.
\end{prop}
Here, the \emph{base} of the parabolic segment is just the chord of the segment. The \emph{height} is the segment $TR$, where $T$ is the midpoint of the base and $R$ is the point where the line parallel to the axis of the parabola through $T$ meets the parabola, see Figure \ref{parabola0}. The line tangent to the parabola in $R$ is parallel to the base. The point $R$ is the \emph{vertex} of the segment \cite[(QP), Proposition 1, p. 234 and Proposition 18, p. 247]{heath}.

Moreover, Archimedes uses the above result at the beginning of his book \emph{On Conoids and Spheroids} (CS)  to prove (see \cite[(CS), Proposition III; t. I, p. 149-152]{verecke}), see Figure \ref{fig:par0}:
\begin{prop}[CS]
\label{prop:3}
If from a same parabola two segments are cut off in any way, and the diameters of the segments are equal, then the segments will have equal area.
\end{prop}
Here, Archimedes calls \emph{diameter} the segment $TR$ that he called \emph{height} in (QP). In fact, $TR$ bisects all the chords parallel to the base of the segment.

\begin{figure}[t]
\label{parabola0}
\begin{center}
\includegraphics[scale=.8]{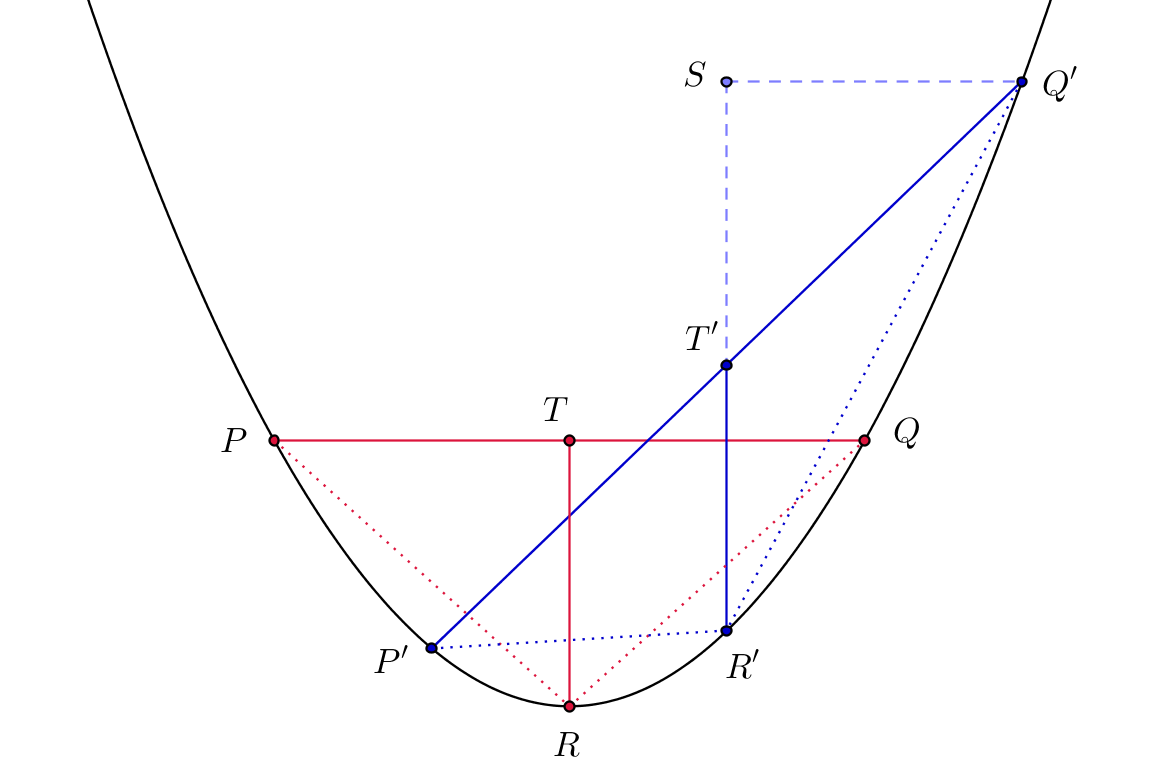}
\caption{$T$ and $T'$ are midpoints of the respective chords. When $TR=T'R'$, then $QT=Q'S$, and the parabolic segments cut off by the chords $PQ$ and $P'Q'$ have equal area.}\label{fig:par0}
\end{center}
\end{figure}

The first propositions in (CS) concern plane geometry, some of them dealing with the quadrature of the ellipse (Propositions 4-6) (we do not know any result by Archimedes related to the quadrature of elliptic or hyperbolic segments). But the main target in (CS) is to show the volumes of the solid segments cut off by a plane from a paraboloid of revolution (or \emph{right-angled conoid}), a two-sheeted hyperboloid of revolution  (or \emph{obtuse-angled conoid}), and an ellipsoid of revolution (or \emph{spheroid}).

In this way, generalizing the main result in (QP), Archimedes proves in (CS) (see \cite[(CS), Propositions 21 and 22, p. 131-133]{heath}):
\begin{prop}[CS]
\label{prop:21-22}
Any segment of a paraboloid of revolution is half as large again as the cone or segment of a cone which has the same base and the same axis.
\end{prop}
The definitions of base and vertex are the suitable adaptations from the one-dimensional case. The axis is the segment joining the vertex with the center of the base. In \cite[(CS), Proposition 23, p. 134-136]{heath}, he also proves the three dimensional version of Proposition III of (CS), namely:
\begin{prop}[CS]
\label{prop:23}
If from a paraboloid of revolution two segments are cut off, and if the axes of the segments are equal, the segments will be equal in volume.
\end{prop}

From Propositions III and 23 of (CS), respectively, two corollaries follow: given a parabola (respectively, a paraboloid) $\mathcal{P}$, any line (plane) tangent to another parabola (paraboloid) $\mathcal{P}'$ which is a translation of $\mathcal{P}$ along its axis of symmetry cuts off from $\mathcal{P}$ a segment of constant area (volume). We find it in the literature as one of the problems of the Cambridge Examinations for the Degree of B. A. around 1800: \emph{Required to prove, that if two parabolas of the same parameter, revolve round their axes, which are in the same straight line, thereby generating two paraboloids one within the other, and a plane be drawn touching the interior paraboloid in any point whatever of its surface, this plane will cut off from the exterior paraboloid a constant volume} \cite[p. 189-190]{wright}.

On the other hand, the reciprocal statements are true, namely: If the straight lines (planes, respectively) tangent to a certain smooth and regular curve (surface) cut off from a given parabola (paraboloid) segments of constant area (volume), then the curve (surface) must be a translation of the
parabola (paraboloid) along its axis of symmetry. This converse result was already known, at least, since the beginning of XIX century. Indeed, we can find the following statement of another problem of the
Cambridge examinations around 1830: \emph{A plane is so moved as always to cut off from a given paraboloid of revolution equal volumes; determine the equation to the surface to which it is always a tangent}
\cite[p. 38-40]{cook}.

Analogous results relative to the conservation of the area (volume) of the
segments cut off from an ellipse (ellipsoid), or from one of the branches (sheets) of an hyperbola (two-sheeted hyperboloid), by any straight line (plane) which is tangent to a certain similar conic (quadric), concentric and homothetic of the former can be also stated and proved. The converse results also hold; therefore, we have characterizations for these other conic sections (quadrics, of revolution or not).

With regard to the reciprocal result we note that, both two-dimensional and three-dimensional cases arise naturally in a geometrical-mechanic context which is not far from Archimedes. In particular, from his two books
\emph{On the floating bodies}, (FB).
Let us think of an homogeneous piece of wood which is paraboloid-shaped and that, due to its shape and the density of wood, can be floating in the water in equilibrium, with the vertex heading the bottom, and only partially submerged. When we move the piece a bit, the variable segment of paraboloid submerged in each position of equilibrium has a constant volume $V$. This is because (we neglect the weight of the air) the weight of the piece must be counterbalanced by the weight of the water displaced, according to the Archimedes' principle of
Hydrostatics.\footnote{\cite[p. 257-259: (FB), Book I, Propositions 5-7]{heath}.
Archimedes studies the stability of the flotation of a right segment of a paraboloid of revolution in Book II of (FB). } Each plane that cuts off from the floating body this volume $V$ is called a \emph{plane of flotation}. The planes of flotation envelop a surface called \emph{surface of flotation} \cite[p. 30-31]{dupin}.  The floating body [the ship] \emph{moves as if this surface rolls and slides on the plane surface of the water}, see \cite[p. 156]{green}.

The converse result of the above states that, while the plane of flotation does not attain the base of the segment of the floating paraboloid, the flotation surface will be a portion of a paraboloid, whose shape is the same as the shape of the floating paraboloid, translated along the axis of the latter.

Ch. Dupin showed in 1814 that the contact point of each plane of flotation with the surface of flotation is the center of gravity of the chordal section (we suppose that this section is homogeneous) intercepted in such plane by the exterior surface of the floating body. In other words, the surface of flotation is the locus of the centers of gravity of the associated chordal sections (see \cite[p. 231-232]{reed} and \cite[p. 125-127, \S{}47]{blas}).

When the outer shape of the floating body is an ellipsoid or a sheet of a two-sheeted hyperboloid, the surface of flotation is a concentric homothetic similar quadric. When the outer shape is a one-sheeted suitably truncated (i.e., the truncation leaves finite volumes) hyperboloid, the surface of flotation is a portion of a sheet of a concentric two-sheeted hyperboloid with the same asymptotic cone, homothetic to the complementary two-sheeted hyperboloid. The surface of flotation of a cone it is also a sheet of a two-sheeted hyperboloid. This is more familiar to us. When the outer shape of the floating body is a paraboloid, the surface of flotation is an identical paraboloid,
translated along the axis \cite[p. 27-35]{bravais} (see also \cite[p. 262, p. 295 ss., p. 343-346]{poldude} and \cite[p. 50-51]{besant}). It is also remarkable that J. Pollard and A. Dudebout \cite[p. 343-344]{poldude} give a proof by synthetic geometry of the qualitative (not metric) part of these results.

In the work by B. Richmond and T. Richmond \cite{richmond}, we find twelve results of characterization of a parabola. The aim of our work is to provide, in a multidimensional context, characterizations of generalized quadrics as hypersurfaces whose tangent planes cut off, from a certain similar hypersurface, compact sets of a constant volume. In this regard, in the present work we try to unify the direct and converse results found in the literature about conics and quadrics, and present them as characterization theorems of generalized quadrics, in a multidimensional context.

We will deal with the constant $(n+1)$-dimensional volumes ($n\ge 1$) of the compact sets cut off by $n$-dimensional hyperplanes from generalized quadrics (paraboloids, two-sheeted hyperboloids, and ellipsoids, but also cones and one-sheeted truncated hyperboloids) in the Euclidean space $E^{n+1}$. Such extensions are inspired by the paper \cite{mgolomb} of M. Golomb.

The procedure used in our proofs is based on $(n+1)$-tuple integration and it seems to be new, or at least, it has not been used in any of the revised results in this introduction.

Throughout the paper, we use the following definitions and notations. We write $(\boldsymbol{x},z)=(x_1,\ldots,x_n,z)$ to denote a generic point of $E^{n+1}$. Given a differentiable function $g$, defined on $E^{n+1}$, we will denote the partial derivatives as $g_{x_i}$, for $i=1,\ldots,n$, and $g_{z}$, when referring to the $(n+1)$-th coordinate. We will say that the nonempty set $\mathcal{M}\subset E^{n+1}$ is a $C^{(2)}$-surface if there exists an open set $A\subset E^{n+1}$ and a real-valued function $g\in C^{(2)}(A)$ such that $\sum_{i=1}^{n}\abs{g_{x_i}(\boldsymbol{x},z)}+\abs{g_z(\boldsymbol{x},z)}>0$ whenever $g(\boldsymbol{x},z)=0$, and $\mathcal{M}=g^{-1}(0)$ (see, for example, \cite[Th. 5-1]{spivak}). When  $g$ is a quadratic polynomial in the variables $x_1$, \ldots, $x_n$, $z$, the surface is called a \emph{quadric}. For example, if $\abs{a_i}>0$ for all $i=1,\ldots,n$, then $z=\sum_{i=1}^n a_i^2x_i^2$ is a \emph{paraboloid}, $z^2+\sum_{i=1}^n a_i^2x_i^2=1$ is an \emph{ellipsoid}, $z^2-\sum_{i=1}^n a_i^2x_i^2=-1$ is a \emph{one-sheeted hyperboloid}, and $z^2-\sum_{i=1}^n a_i^2x_i^2=1$ is the \emph{complementary} \emph{two-sheeted hyperboloid}, both with asymptotic cone $z^2-\sum_{i=1}^n a_i^2x_i^2=0$. The \emph{upper sheet} of the latter hyperboloid is the surface contained into the half-space $\{z\ge 1\}$.

When $\mathcal{M}$ is a $C^{(2)}$-surface
such that $\mathcal{M}=\partial K$, where the set $K\subset E^{n+1}$, with nonempty interior, is compact and convex, we will call $\mathcal{M}$ a $C^{(2)}$-\emph{ovaloid}.

Finally, by $V_m(\mathcal{S})$ we denote the $m$-dimensional volume of the compact set $\mathcal{S}$.

The structure of the paper is as follows. In Section \ref{sec:parabolas} we present the preliminary plane parabolic case ($n=1$). In the sequel sections we will show multidimensional extensions for the parabolic, two-sheeted hyperbolic, one-sheeted hyperbolic, conic, and elliptic cases. In the two-sheeted hyperbolic and elliptic cases we also extend the direct results proved by Archimedes in (CS). Some of the computations will be moved to a final appendix.


\section{Characterizing a parabola}
\label{sec:parabolas}

As we said in the introduction, the direct part of the following theorem is an immediate corollary of Proposition \ref{prop:3}. It was already presented by the first author in \cite{O}. In this section, we provide a characterization of the parabola. We mentioned earlier that other characterizations of the parabola were already supplied in \cite{richmond}. We denote by $(x,z)$ the rectangular plane coordinates.

\begin{teor} Let $p,k\in\mathbb{R}$ and $pk>0$. Every straight line tangent to the parabola
$\mathcal{P}'\colon z=p^2x^2+k^2$  cuts off from the parabola $\mathcal{P}\colon z=p^2x^2$ a parabolic segment  of constant
area equal to
$\frac{4}{3}\frac{k^3}{p}$   (Figure~\ref{fig:parabola1}).

\begin{figure}[ht]
\begin{center}
\includegraphics[scale=.8]{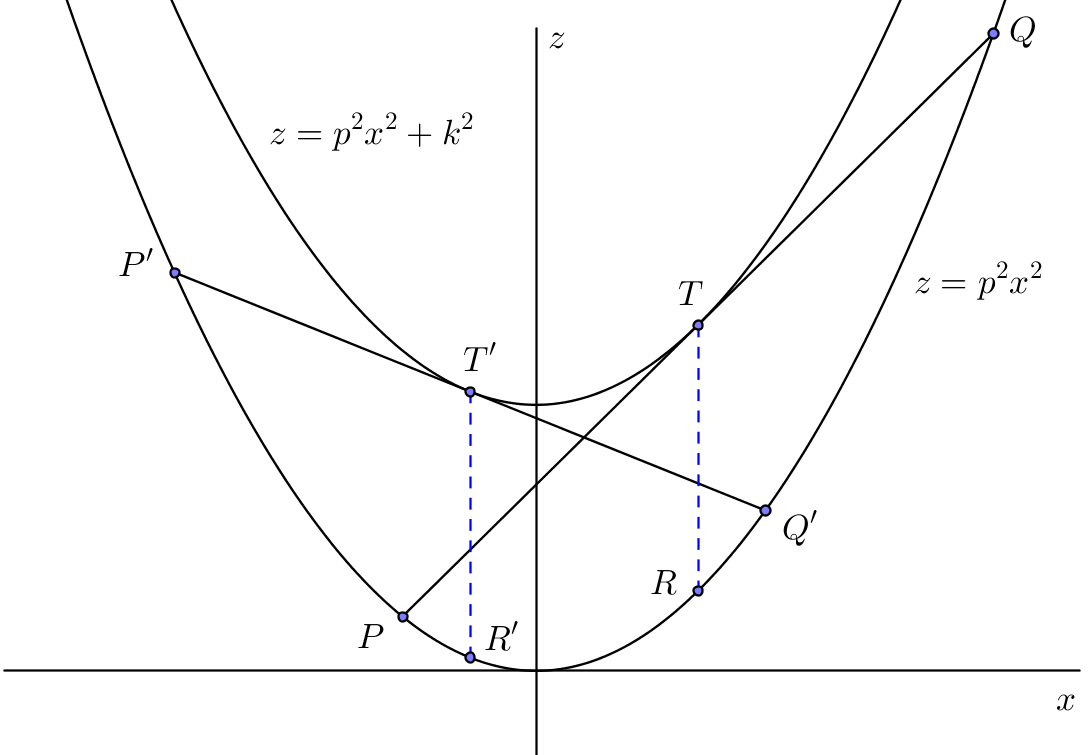}
\caption{The areas of the segments cut off from the parabola $z=p^2x^2$ by the chords $PQ$ and $P'Q'$, tangent to the translated parabola $z=p^2x^2+k^2$ respectively at the points $T$ and $T'$, are equal.}\label{fig:parabola1}
\end{center}
\end{figure}

Conversely, let $f\in C^{(2)}(\mathbb{R})$ be such that
$f(x)>p^2x^2$ for all $x\in\mathbb{R}$. If the tangent line  to the curve $z=f(x)$ at each one of its points cuts off from the parabola $\mathcal{P}$ a segment of constant area $A$, then $f(x)=p^2x^2+k^2$, where $k=\sqrt[3]{3pA/4}$.
\end{teor}

\begin{proof}

Let $T(\overline{x},\overline{z}+k^2)$ be any point of the parabola $\mathcal{P}'$.  Let $\mathcal{S}$ be the parabolic segment which the tangent line to $\mathcal{P}'$ at the point $T$ cuts off from the parabola $\mathcal{P}$. The \emph{vertex} of  $\mathcal{S}$ is the point  $R(\overline{x},\overline{z})$. We have $\abs{TR}=k^2$, a constant, and thus, by Proposition \ref{prop:3}, the area of  $\mathcal{S}$ is constant (independent of the point $T$). We may know the value of this constant by computing, for example, the area of the \emph{right} segment $\mathcal{S}_0$ corresponding to $T(0,k^2)$, using Proposition \ref{prop:mainQP}. Then we have (see Figure \ref{fig00})

\begin{figure}[ht]
\begin{center}
\includegraphics[scale=.8]{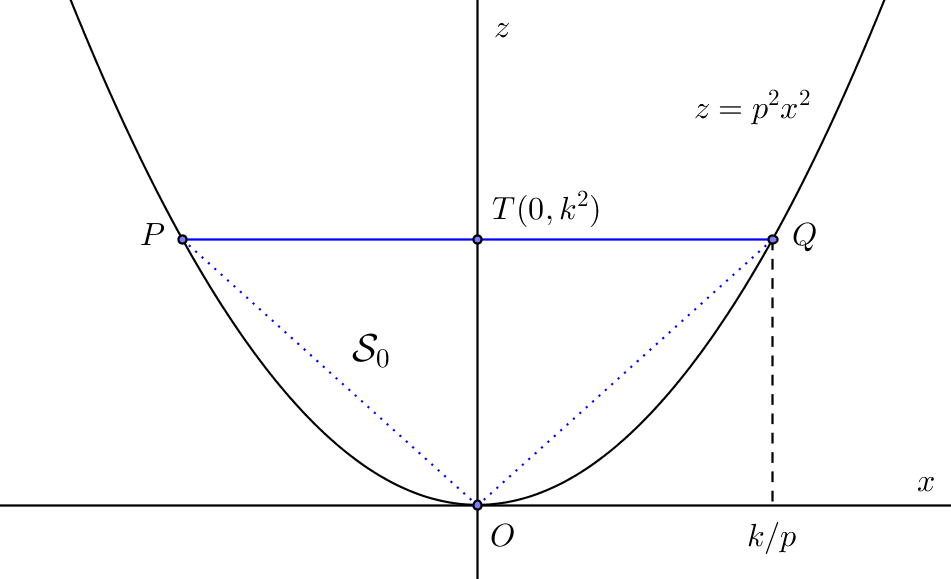}
\caption{}\label{fig00}
\end{center}
\end{figure}

$$
V_2(\mathcal{S})=\frac43\cdot\text{area$(\triangle POQ)$}=\frac{4}{3}\frac{k^3}{p}\,.
$$

Conversely, let us suppose that the straight line $\ell$, tangent to the curve $z=f(x)$ at a generic point $L(\overline{x},\overline{z})$, where $\overline{z}=f(\overline{x})$, cuts off from the parabola $\mathcal{P}$ a parabolic segment of constant  (independent of $\overline{x}$) area $A>0$. Let $k=\sqrt[3]{3A/4}$, and
write $f'(\overline{x})=\overline{m}$.

\begin{figure}[ht]
\begin{center}
\includegraphics[scale=.8]{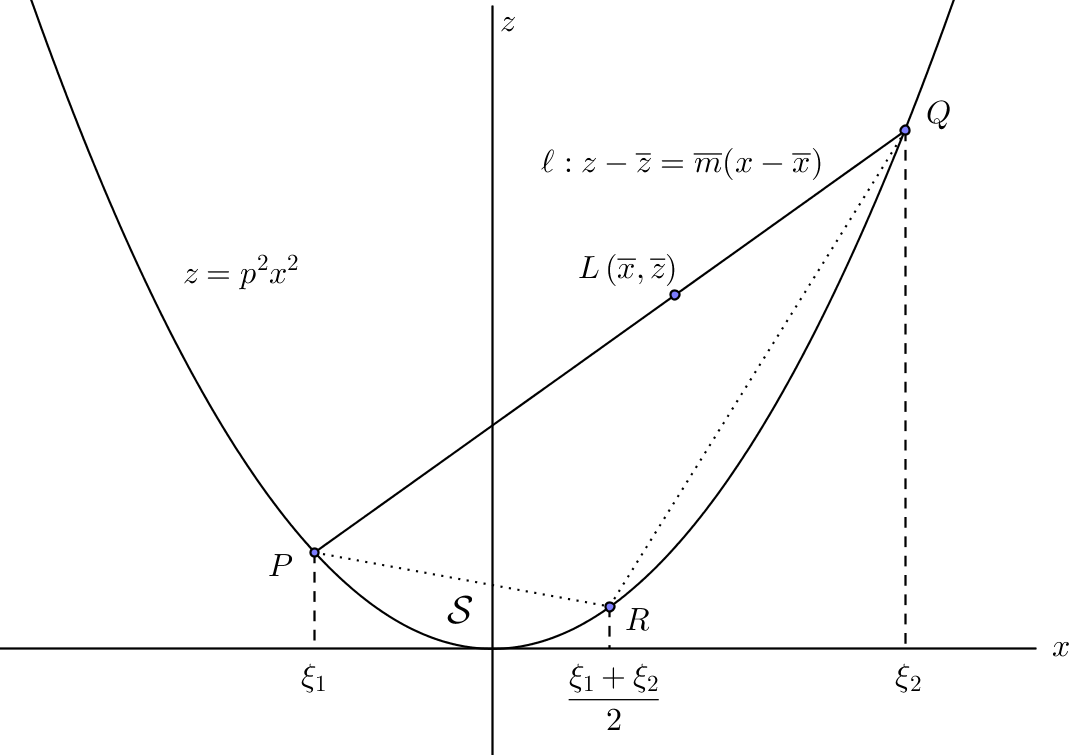}
\caption{The area of the  parabolic segment is $4/3$ of the area of $\triangle PRQ$.}\label{fig2}
\end{center}
\end{figure}

The condition $\overline{z}>p^2\overline{x}^2$ assures that the line $\ell$ meets $\mathcal{P}$ in two points $P$ and $Q$ (see Figure \ref{fig2}). Furthermore, if the absciss\ae{} of $P$ and $Q$ are  $\xi_1$ and $\xi_2$ ($\xi_2\ge \xi_1$), we have
$$
\xi_2-\xi_1=2\sqrt{\overline{z}-\overline{m}\,\overline{x}+\frac{\overline{m}^2}4}.
$$

The \emph{vertex} of the parabolic segment $\mathcal{S}$ which $\ell$ cuts off from $\mathcal{P}$ is the point $R\in \mathcal{P}$ which has abscissa $(\xi_2+\xi_1)/2=\overline{m}/2$. Then, the area of $\mathcal{S}$ is, according to  Proposition \ref{prop:mainQP},

\begin{align*}
A&=\frac23\cdot
\begin{vmatrix}
1&\xi_2 &\xi_2^2\\
1&\xi_1 &  \xi_1^2  \\
1&\frac{\xi_1+\xi_2}2 &
\big(\frac{\xi_1+\xi_2}2\big)^2
\end{vmatrix}    \\
&=\frac16(\xi_2-\xi_1)^3\\
&=\frac43\left(\overline{z}-\overline{m}\,\overline{x}+\frac{\overline{m}^2}4\right)^{3/2}.
\end{align*}

And, since $A=\frac{4}{3}k^3$, we arrive at the equation
$$
\overline{z}-\overline{m}\,\overline{x}+\frac{\overline{m}^2}4=k^2,
$$
or
\begin{equation}
\label{ec:ode}
f(\overline{x})-\overline{x}f'(\overline{x})+\frac14(f'(\overline{x}))^2-k^2=0,
\end{equation}
which must be satisfied by the unknown function $f(\overline{x})$ for all $\overline{x}\in\mathbb{R}$. This is a Clairaut's first order ordinary differential equation.

\begin{figure}[ht]
\begin{center}
\includegraphics[scale=.8]{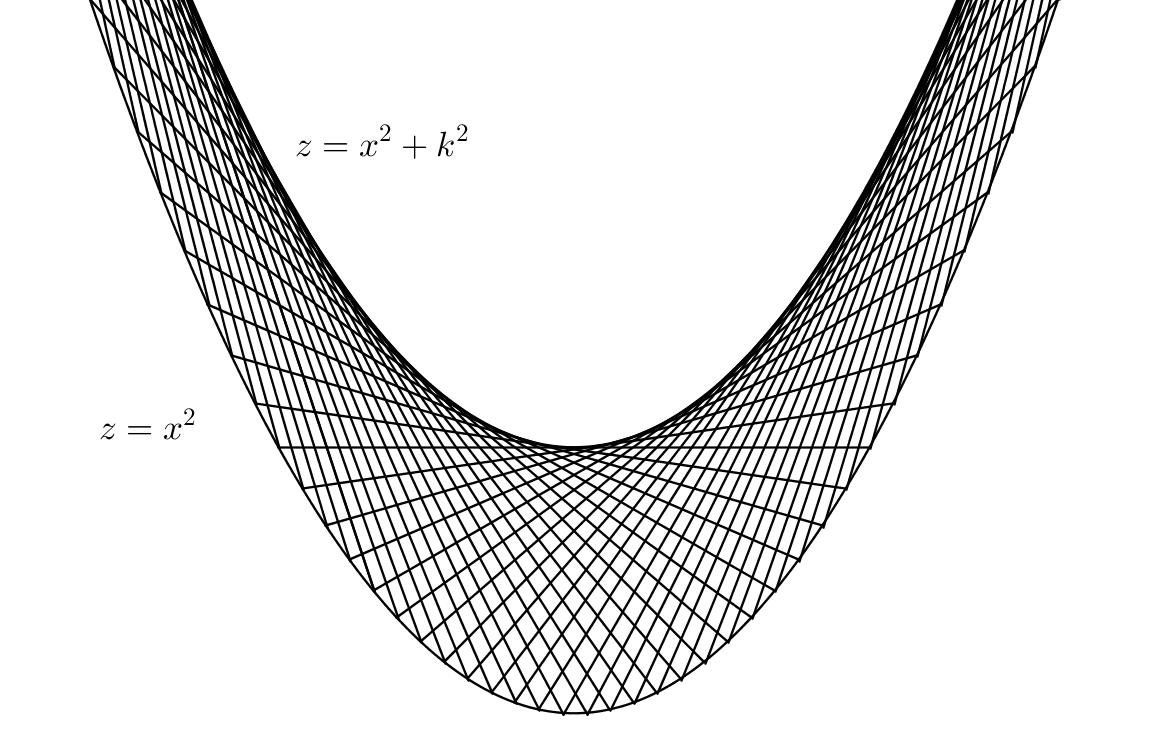}
\caption{The parabola $z=x^2$, some lines of the associated family $\phi(x,z,\lambda)=0$, and the envelope of this family of lines, the parabola  $z=x^2+k^2$.}\label{fig:parabola2}
\end{center}
\end{figure}

The one-parameter family of lines solution of \eqref{ec:ode}
\begin{equation}
\label{eq:familia}
\phi(x,z,\lambda):=
z-\lambda x+\frac{\lambda^2}{4}-k^2=0
\end{equation}
is called the \emph{general solution}. The differential equation admits only another solution, the \emph{singular integral}. This is the envelope of the family \eqref{eq:familia}, namely,
the equation which is obtained by eliminating $\lambda$ from the parametric equations
\[
\phi(x,z,\lambda)=0,\qquad  \frac{\partial{\phi(x,z,\lambda)}}{\partial \lambda}=0.
\]
Then, we get the translated parabola
\[
z=f(x)=x^2+k^2
\]
as the only valid solution for the converse statement (Figure~\ref{fig:parabola2}).
\end{proof}


\section{The general parabolic case}
\label{Sec:paraboloide}

The \emph{base} of an $(n+1)$-dimensional segment of paraboloid is the chordal section taken in the inner part of a paraboloid by the hyperplane that cuts off the paraboloid. In general, it will be an $(n+1)$-dimensional ellipsoid. The
\emph{vertex} is the contact point with the paraboloid of the hyperplane which is parallel to the \emph{base} and tangent to the paraboloid. The \emph{axis} is the straight segment which joins the vertex with the center of the ellipsoid in the base.
The lemma which follows can also be obtained by $n$-tuple integration \cite[p. 139-140]{mgolomb}.\footnote{In \cite[Th. 2]{mgolomb}, S. Golomb shows, establishing the $(n+1)$-dimensional extension of the above cited results, that the Archimedean  factors $4/3$ for $n=1$, and $3/2$ for $n=2$, generalize to $\frac{2(n+1)}{n+2}$.}

\begin{lema} \label{lema2}
Let $n\ge 1$. Let $(\overline{p_1},\ldots,\overline{p_n})\in\mathbb{R}^n$ and $P=(\overline{\boldsymbol{x}},\overline{z})\in\mathbb{R}^{n+1}$  be fixed.
Suppose that $\overline{z}> \sum_{i=1}^n\overline{x_i}^2$. The Euclidean volume of the compact set $\mathcal{S}\subset E^{n+1}$, cut off by the hyperplane
$
\Upsilon\colon z-\overline{z}=\sum_{i=1}^n\overline{p_i}(x_i-\overline{x_i})
$
from the paraboloid $z=\sum_{i=1}^n x_i^2$, is equal to
$
F_0\left(\zeta_P\right),
$
where $\zeta_P=\overline{z}-\sum_{i=1}^n\overline{p_i}\,\overline{x_i}+\sum_{i=1}^n\frac{\overline{p_i}^2}4$ and
\begin{equation}\label{efe0}
F_0(t)=\frac{\pi^{n/2}}{\Gamma(n/2+2)}\,t^{n/2+1}.
\end{equation}
In fact, this volume is equal to $\frac{2(n+1)}{n+2}$ times the volume of the cone $\mathcal{K}$ which has the same \emph{base} and the same \emph{axis} as $\mathcal{S}$.
\end{lema}

\begin{proof}
The volume of $\mathcal{S}$ is given by the $(n+1)$-tuple integral
\begin{equation*}
V_{n+1}(\mathcal{S})=\int_{\mathcal{S}} \,d\boldsymbol{x}\,dz.
\end{equation*}

Change the variables according to the following affine transformation (see Figure \ref{paraboloide2})
\begin{equation}
\begin{pmatrix} x_1\\ x_2\\\vdots\\ x_n\\
z\end{pmatrix}=
\begin{pmatrix} \frac{\overline{p_1}}2\\ \frac{\overline{p_2}}2\\\vdots\\ \frac{\overline{p_n}}2\\
\frac{\sum_{1=1}^n\overline{p_i}^2}4\end{pmatrix}+
\begin{pmatrix}
1&0&\cdots&0&0\\
0&1&\cdots&0&0\\
\vdots&\vdots&\ddots&\vdots&\vdots\\
0&0&\cdots&1&0\\
\overline{p_1}&\overline{p_2}&\cdots&\overline{p_n}&1\end{pmatrix}
\begin{pmatrix}
y_1\\y_2\\\vdots\\y_n\\\zeta
\end{pmatrix}.
\end{equation}

\begin{figure}[ht]
\begin{center}
\includegraphics[scale=.8]{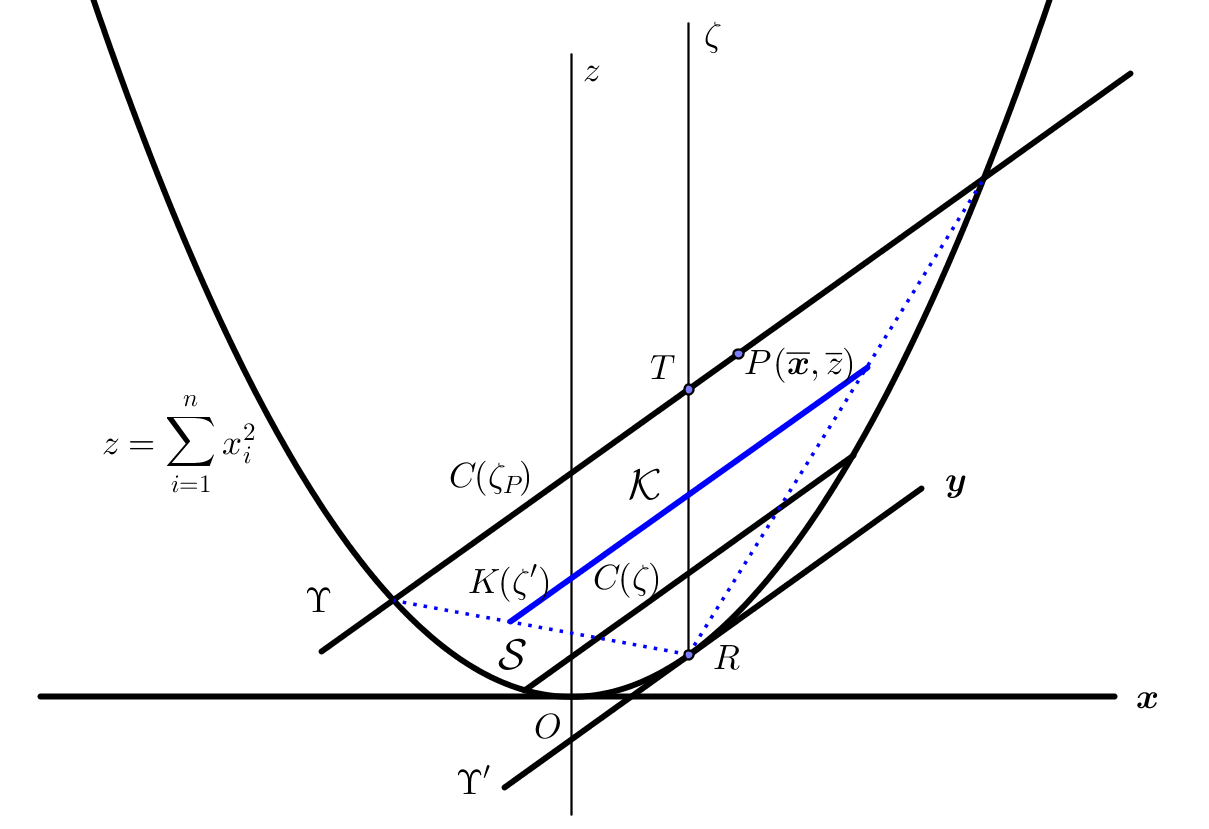}
\caption{The two-dimensional plane defined by the axis of symmetry of the paraboloid and  the normal vector of the hyperplane~$\Upsilon$ (a scheme).} \label{paraboloide2}
\end{center}
\end{figure}

The equation of the hyperplane $\Upsilon$ in the new variables is $\zeta=\zeta_P$ and the equation of the paraboloid turns in $\zeta=\sum_{i=1}^n y_i^2$, so that, for each fixed value of $\zeta$, the chordal section $C(\zeta)$ in the transformed $(\boldsymbol{y},\zeta)$-domain $\mathcal{S}^*$ is a $n$-sphere of radius $\zeta^{1/2}$.

The $\zeta$-coordinate is $0$ for the hyperplane $\Upsilon'$ parallel to $\Upsilon$ and which is tangent to the paraboloid at the point $R$, origin of the $(\boldsymbol{y},\zeta)$-coordinates.

The Jacobian of the transformation is
$\frac{\partial(\boldsymbol{x},z)}{\partial(\boldsymbol{y},\zeta)}=1$.
On the other hand, the $n$-dimensional volume of the $n$-sphere of radius $r>0$ is given by the formula
\begin{equation}
\label{eq:volumen}
\alpha_n(r)=\frac{\pi^{n/2}}{\Gamma(n/2+1)}\cdot r^n.
\end{equation}
Then, writing $V_n^*(L)$ for the $n$-volume of the region $L$ in the $(\boldsymbol{y},\zeta)$-space,
\begin{align*}
V_{n+1}(\mathcal{S})&=\int_{\mathcal{S}^*} \,d\boldsymbol{y}\,d\zeta=\int_0^{\zeta_P} V_n^*(C(\zeta))\,d\zeta=\alpha_n(1)\int_0^{\zeta_P}\zeta^{n/2}\,d\zeta\\
&=\frac{2\pi^{n/2}}{(n+2)\Gamma(n/2+1)}\,\zeta_P^{n/2+1} =F_0(\zeta_P),
\end{align*}
as stated.

The Archimedean cone $\mathcal{K}$ has base $C(\zeta_P)$ and vertex $R$. We have
\begin{equation*}
V_{n+1}(\mathcal{K})=\int_{\mathcal{K}} \,d\boldsymbol{x}\,dz=\int_{\mathcal{K}^*} \,d\boldsymbol{y}\,d\zeta=\int_0^{\zeta_P} V_n^*(K(\zeta))\,d\zeta,
\end{equation*}
where $K(\zeta)$ are the $n$-dimensional sections of $\mathcal{K}$ which are parallel to the base $C(\zeta_P)=K(\zeta_P)$. From elementary geometry we have
\[
\frac{V_n^*(K(\zeta))}{V_n^*(C(\zeta_P))}=\Bigl(\frac{\zeta}{\zeta_P}\Bigr)^n,
\]
so that
\begin{align*}
V_{n+1}(\mathcal{K})&=V_n^*(C(\zeta_P))\frac{\zeta_P}{n+1}=\frac{\pi^{n/2}}{\Gamma(n/2+1)}\,\zeta_P^{n/2}\cdot\frac{\zeta_P}{n+1}\\
&=\frac{\pi^{n/2}}{(n+1)\Gamma(n/2+1)}\,\zeta_P^{n/2+1},
\end{align*}
and this yields
\begin{equation*}
V_{n+1}(\mathcal{S})=\frac{2(n+1)}{n+2}\cdot V_{n+1}(\mathcal{K}),
\end{equation*}
(already stated by Archimedes for $n=1$ and $n=2$).
\end{proof}

The direct case of the next result for $n=1$ and $n=2$ (and for a paraboloid of revolution) is due to Archimedes, as we noted above.
And the converse result, for $n=2$, is equivalent to find the \emph{surface of flotation} of a floating  paraboloid. A.~Bravais
\cite[p. 34-35]{bravais} already showed that it was an identical paraboloid.

\begin{teor}
\label{theor:par}
Let $n\ge 1$. Let $p_1$, \ldots, $p_n>0$,
and $k>0$ be fixed, and $p(\boldsymbol{x})=\sum_{i=1}^n p_i^2x_i^2$ for all
$\boldsymbol{x}\in\mathbb{R}^n$. Every hyperplane tangent to the paraboloid
$z=p(\boldsymbol{x})+k^2$ cuts off from the paraboloid $z=p(\boldsymbol{x})$ a $(n+1)$-dimensional compact set $\mathcal{S}$ of constant Euclidean volume equal to $F(k^2)=p_1\cdots
p_n F_0(k^2)$.

Conversely, let $f\in C^{(2)}(\mathbb{R}^n)$ be a real-valued function such that
\begin{equation}
\label{eq:condmul1}
f(\boldsymbol{x})> p(\boldsymbol{x})
\quad \text{ for all $\boldsymbol{x}\in\mathbb{R}^n$.}
\end{equation}
If every hyperplane tangent to the $C^{(2)}$-surface $z=f(\boldsymbol{x})$ cuts off from
the paraboloid $z=p(\boldsymbol{x})$
a compact set of constant $(n+1)$-dimensional Euclidean volume $V$,  then $f(\boldsymbol{x})=p(\boldsymbol{x})+k^2$, where $k^2=F^{-1}(V)$.
\end{teor}

\begin{proof}
Without loss of generality we can consider $p_i=1$ for all $i=1,\ldots,n$. The linear transformation $x_i=p_iX_i$ ($i=1,\ldots,n$), $z=Z$ turns the paraboloids $z=\sum_{i=1}^nx_i^2$ and $z=\sum_{i=1}^nx_i^2+k^2$, respectively, into the paraboloids $Z=p(\boldsymbol{X})$ and $Z=p(\boldsymbol{X})+k^2$ and the theorem follows from this special case.

For the direct result, just consider $\Upsilon$
the hyperplane tangent to the paraboloid $z=\sum_{i=1}^nx_i^2+k^2$ at the generic point
$T(\overline{\boldsymbol{x}},\overline{z})$,
$$
\Upsilon\colon z=\overline{z}+2\sum_{i=1}^n \overline{x_i}(x_i-\overline{x_i})=k^2+2\sum_{i=1}^n \overline{x_i} x_i-\sum_{i=1}^n \overline{x_i}^2.
$$
By applying Lemma \ref{lema2} (we have in this case $\overline{p_i}=2\overline{x_i}$ and $\overline{z}=\sum_{i=1}^n \overline{x_i}^2+k^2$; the $(\boldsymbol{x},z)$-coordinates of the point $R$ are $(\overline{\boldsymbol{x}},\overline{z}-k^2)$), it results that the Euclidean volume of the compact set $\mathcal{S}$ is $F_0(k^2)$, as stated (observe that, for example, $V_3(\mathcal{S})=\frac\pi2\,k^4$).

Conversely, assume that
every hyperplane tangent to the $C^{(2)}$-surface unknown $z=f(\boldsymbol{x})$ cuts off from
the paraboloid $z=\sum_{i=1}^nx_i^2$,
a compact set of constant $(n+1)$-dimensional Euclidean volume $V$. Let $k^2=F_0^{-1}(V)$.

Let $P=(\overline{\boldsymbol{x}},\overline{z})$ ($\overline{z}=f(\overline{\boldsymbol{x}})$) be an arbitrary point of the surface $\mathcal{M}\colon z=f(\boldsymbol{x})$. We are assuming that $\sum_{i=1}^n\abs{f_{x_i}(\overline{\boldsymbol{x}})}\ne 0$ and $f(\overline{\boldsymbol{x}})>\sum_{i=1}^n\overline{x_i}^2$. The tangent hyperplane to $\mathcal{M}$ at the point
$P$ is
$$
\Upsilon_P\colon z-\overline{z}=\sum_{i=1}^n\overline{p_i}(x_i-\overline{x_i}),
$$
where we have written $f_{x_i}(\overline{\boldsymbol{x}})=\overline{p_i}$ for $i=1,\ldots,n$.   By applying Lemma \ref{lema2} it results that
$$
V=F_0\left(\overline{z}-\sum_{i=1}^n\overline{p_i}\,\overline{x_i}+\sum_{i=1}^n\frac{\overline{p_i}^2}4\right).
$$
But we had also
$V=F_0(k^2)$. As $F_0$ is one-to-one, we get
\begin{gather}
\overline{z}=\sum_{i=1}^n\overline{p_i}\,\overline{x_i}+k^2-\sum_{i=1}^n\frac{\overline{p_i}^2}4.
\label{eq:claigen}
\end{gather}

In this manner, since $\overline{\boldsymbol{x}}$ runs over $\mathbb{R}^n$, the condition which we have obtained for the unknown $C^{(2)}$-surface
$z=f(\boldsymbol{x})$ is the Clairaut's first order partial differential equation
\begin{equation}
f(\boldsymbol{x})=\sum_{i=1}^n f_{x_i}(\boldsymbol{x})\,x_i+k^2-\frac14\sum_{i=1}^n\bigl(f_{x_i}(\boldsymbol{x})\bigr)^2, \label{eq:claigen2}
\end{equation}
of which the \emph{complete integral}\footnote{See, for example,
\cite[p. 94-95]{courant}. A $n$-parameter family $z=\phi(x_1,\ldots,x_n;a_1,\ldots, a_n)$ of solutions of a first order partial differential equation is called a \emph{complete integral} of the equation in a region of the $(x_1,\ldots,x_n)$-space if in the region considered the rank of the matrix
$$
M=\begin{pmatrix}
\phi_{a_1}&\phi_{x_1a_1}&\cdots&\phi_{x_na_1}\\
\phi_{a_2}&\phi_{x_1a_2}&\cdots&\phi_{x_na_2}\\
\vdots&\vdots&\ddots&\vdots\\
\phi_{a_n}&\phi_{x_1a_n}&\cdots&\phi_{x_na_n}
\end{pmatrix}
$$
is $n$. In our case, the matrix $M$ is the identity matrix for all  $\boldsymbol{x}\in\mathbb{R}^n$.
}
is formed by the $n$-parameter family of hyperplanes
\begin{equation}
\label{eq:intcom2}
z=\sum_{i=1}^n a_ix_i + k^2-\frac14\sum_{i=1}^n a_i^2.
\end{equation}

The singular integral of  \eqref{eq:claigen2} is the translated paraboloid $z=\sum_{i=1}^n x_i^2+k^2$, envelope of the $n$-parameter family of hyperplanes \eqref{eq:intcom2}. It is obtained by eliminating the $n$ parameters $a_i$ from the equations
\[
\left\{\begin{array}{l}
z=\sum_{i=1}^n a_ix_i + k^2-\frac14\sum_{i=1}^n a_i^2,\\[4pt]
x_i=\frac {a_i}2, \quad i=1,\ldots,n.
\end{array}\right.
\]
\end{proof}

\begin{rem}
Apart from the complete integral and the paraboloid $z=\sum_{i=1}^n x_i^2+k^2$, equation \eqref{eq:claigen2} also admits as solution (see \cite[p. 103-105]{courant}) the envelope of any arbitrary $r$-parameter subfamily of hyperplanes
($1\le r\le n-1$) selected from \eqref{eq:intcom2} by means of arbitrary functions ($a_{i}=\varphi(t_1,\ldots,t_r)$ for $i=1,\ldots,n$, say). But these type of envelopes are, in general, $n$-manifolds \emph{generated by} linear $(n-r)$-manifolds (or $(n-r)$-\emph{flats}, see \cite[p. 39]{gupta}) that will pass through the paraboloid $z=\sum_{i=1}^n x_i^2$, and will not verify the stated condition \eqref{eq:condmul1}

In particular, for $n=2$ ($\boldsymbol{x}=(x,y)$), a \emph{complete integral} of \eqref{eq:claigen2} in $\mathbb{R}^2$ consists of the two-parameter family of planes
\begin{equation}
\label{eq:intcom}
z=ax+by+k^2-\frac{a^2}4-\frac{b^2}4.
\end{equation}

The envelope of any one-parameter subfamily of \eqref{eq:intcom} obtained considering, say, $b=\varphi(a)$ where $\varphi$ is any function twice continuously differentiable, that is, the surface defined by
\[
\left\{\begin{array}{l}
z=ax+y\varphi(a)+k^2-\frac{a^2}4-\frac{(\varphi(a))^2}4,\\[4pt]
0=x+y\varphi'(a)-\frac{a}2-\frac{\varphi(a)\varphi'(a)}2,
\end{array}\right.
\]
is also a solution of  \eqref{eq:claigen2}. These envelopes form in this case the \emph{general solution} of the equation. All these surfaces are developable, that is, ruled surfaces (in general, cylinders, cones or tangential developables, see for example \cite[p. 204-208]{geo}) with the same tangent plane along each one of its ruling lines. Anyway, the rulings generating these types of surfaces will pass through the paraboloid $z=x^2+y^2$, and the stated condition \eqref{eq:condmul1} will be not verified. On the other hand, surfaces defined in all $\mathbb{R}^2$ formed by pasting pieces of the types above, can be inside the paraboloid $z=x^2+y^2$, but  will fail the smoothness condition along the lines of paste.

The \emph{singular integral}, envelope of the two-parameter family of planes \eqref{eq:intcom}, obtained by eliminating both parameters $a$ and $b$ from the equations
\[
\left\{\begin{array}{l}
z=ax+by+k^2-\frac{a^2}4-\frac{b^2}4,\\[4pt]
x=\frac a2,\\[4pt]
y=\frac b2,
\end{array}\right.
\]
is the translated paraboloid $z=x^2+y^2+k^2$, the only valid solution of the converse statement when $n=2$.
\end{rem}

\section{Hyperbolic surfaces and cones }
In this section we analyze the surfaces of flotation for two-sheeted hyperboloids, one-sheeted hyperboloids, and also for the cones. In all the cases we use a similar idea to obtain the volume of the region limited by the corresponding quadric and some hyperplanes.

\subsection{The general two-sheeted hyperbolic case}

As in the case of the paraboloid, Archimedes defined the
\emph{vertex} of a hyperboloidal segment as the contact point with the hyperboloid of the hyperplane which is parallel to the \emph{base} and tangent to the hyperboloid, see Figure~\ref{fig:hipermulti4}. The \emph{axis} of the segment is the portion cut off within the segment from the line drawn through the \emph{vertex} of the segment and the vertex of the asymptotic cone enveloping the hyperboloid, i.e., the center of the hyperboloid. The \emph{semidiameter} of the hyperboloid through $R$  is the distance between such vertices \cite[p. 101]{heath}.
Archimedes proves in (CS) (see \cite[p. 136-140: (CS), Propositions 25-26]{heath}):
\begin{prop}
\label{prop:25-26}
If $R$ is the vertex and $\abs{RT}$ the length of the axis of any segment cut off by a plane from a sheet of a two-sheeted hyperboloid of revolution, and if $\abs{OR}$ is the semidiameter of the hyperboloid through $R$, with $OR$ in the same straight line of $RT$, then
\[
\frac{V(\mathcal{S})}{V(\mathcal{K})}\\
=\frac{\abs{RT}+3\abs{OR}}{\abs{RT}+2\abs{OR}},
\]
where $V(\mathcal{S})$ and $V(\mathcal{K})$ denote, respectively, the volumes of the segment and of the cone or segment of a cone with the same base and axis.
\end{prop}

Our next lemma is an extension of this Archimedes result and it gives the volume of the region limited by a two-sheeted hyperboloid and a hyperplane in the Euclidean space $E^{n+1}$.

\begin{lema} \label{lema3}
Let $n\ge 1$. Let $(\overline{p_1},\ldots,\overline{p_n})\in\mathbb{R}^n$ and $P=(\overline{\boldsymbol{x}},\overline{z})\in\mathbb{R}^{n+1}$  be fixed.
Let $d= \overline{z}-\sum_{i=1}^n\overline{p_i}\,\overline{x_i}$ and $q=1-\sum_{1=1}^n\overline{p_i}^2$. Suppose that\footnote{With these conditions, the hyperplane $\Upsilon$ cuts off from the upper sheet of the hyperboloid $z^2=1+\sum_{i=1}^n x_i^2$ a compact set. } $q>0$ and $d\ge \sqrt{q}$.
The Euclidean volume of the compact set $\mathcal{S}\subset E^{n+1}$ cut off by the hyperplane
$
\Upsilon\colon z-\overline{z}=\sum_{i=1}^n\overline{p_i}(x_i-\overline{x_i})
$
from the set $\mathcal{H}=\{ (\boldsymbol{x},z)\vert z\ge 1;\ \sum_{i=1}^n x_i^2\le z^2-1\}$ is equal to
$
G_0(\zeta_P),
$
where $\zeta_P=d/\sqrt{q}$ and
\begin{equation}\label{ge0}
 G_0(t)=\frac{\pi^{n/2}}{\Gamma(n/2+1)}\int_1^t (\zeta^2-1)^{n/2}\,d\zeta.
\end{equation}

In fact, this volume is equal to $\frac{n+1}{(\zeta_P^2-1)^{n/2}(\zeta_P-1)}\int_1^{\zeta_P}(\zeta^2-1)^{n/2}\,d\zeta$ times the volume of the cone which has the same base and the same axis as the segment.
\end{lema}

\begin{proof}
The $(n+1)$-dimensional volume of  $\mathcal{S}$ is
$$
V_{n+1}(\mathcal{S})=\int_{\mathcal{S}}\,d\boldsymbol{x}\,dz.
$$
The hyperplane parallel to $\Upsilon$ and tangent to $\partial\mathcal{H}$ has the equation
\[
\Upsilon'\colon z=\sum_{i=1}^n\overline{p_i}x_i+\sqrt{q},
\]
and the tangency point of $\Upsilon'$ and $\partial\mathcal{H}$  is
$$
R=\left(\frac{\overline{p_1}}{\sqrt{q}}, \ldots, \frac{\overline{p_n}}{\sqrt{q}},\frac1{\sqrt{q}}\right),
$$
as can be easily checked.

Change the variables according to the following linear transformation
\begin{equation} \label{eq10}
\begin{pmatrix} x_1\\ x_2\\\vdots\\ x_n\\
z\end{pmatrix}=
\begin{pmatrix}
1&0&\cdots&0&\frac{\overline{p_1}}{\sqrt{q}}\\[4pt]
0&1&\cdots&0&\frac{\overline{p_2}}{\sqrt{q}}\\
\vdots&\vdots&\ddots&\vdots&\vdots\\
0&0&\cdots&1&\frac{\overline{p_n}}{\sqrt{q}}\\[4pt]
\overline{p_1}&\overline{p_2}&\cdots&\overline{p_n}&\frac1{\sqrt{q}}\end{pmatrix}
\begin{pmatrix}
y_1\\y_2\\\vdots\\y_n\\\zeta
\end{pmatrix}.
\end{equation}
The new $\boldsymbol{y}$-hyperplane is the parallel to $\Upsilon$ through the origin,
and the $\zeta$-axis is the ray $OR$ (see Figure~\ref{fig:hipermulti4}).

\begin{figure}[h]
\begin{center}
\includegraphics[scale=.8]{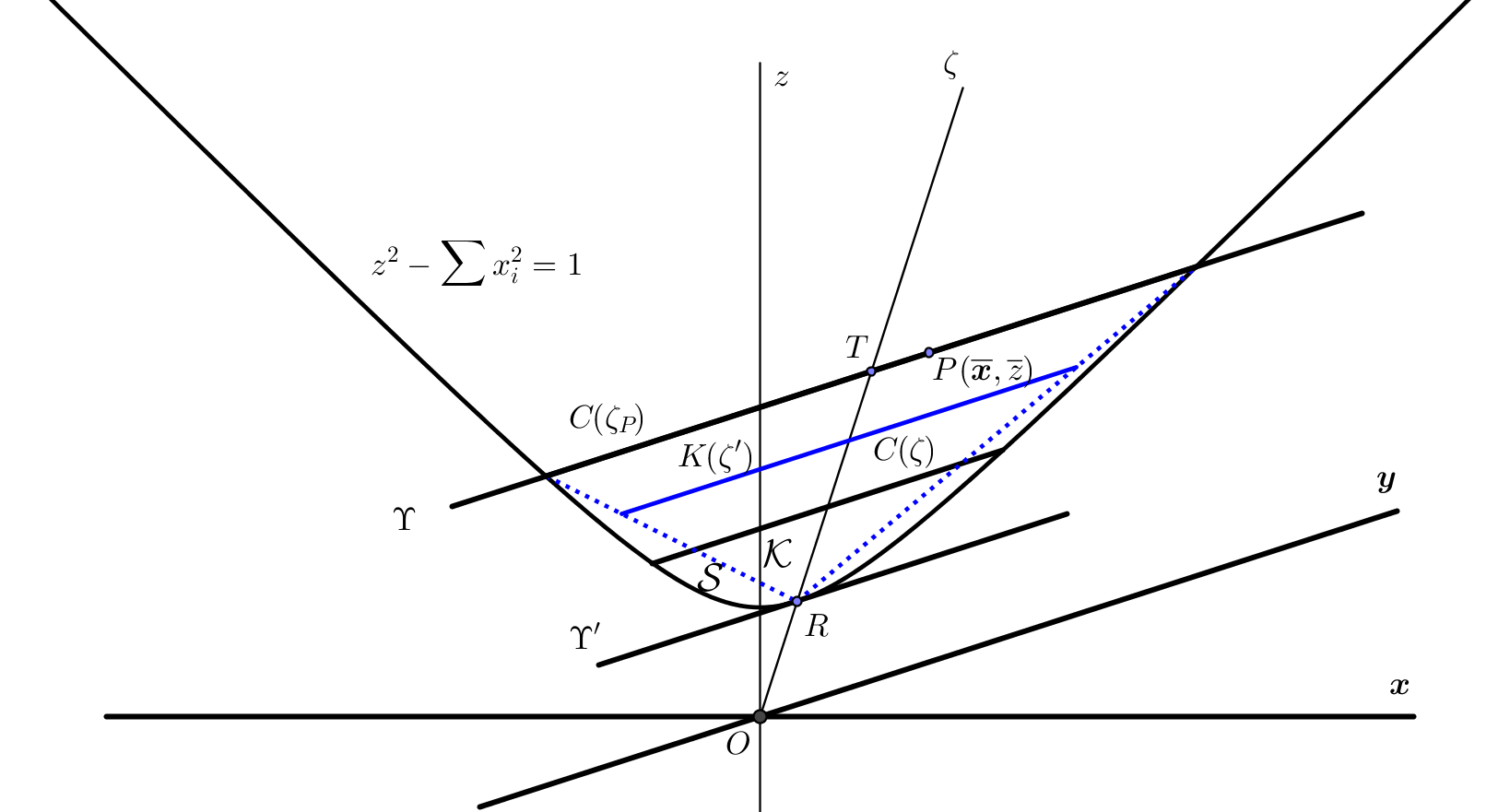}
\caption{}\label{fig:hipermulti4}
\end{center}
\end{figure}

We have (see Note~\ref{app1} in the Appendix) 
\begin{equation}
\label{hypapp3}
\frac{\partial(\boldsymbol{x},z)}{\partial(\boldsymbol{y},\zeta)}=\frac1{\overline{q}}\, J_{n+1}(\overline{p_1},\ldots,\overline{p_n})
=\sqrt{q},
\end{equation}
and we can easily check that the hyperplanes  $\Upsilon'$ and $\Upsilon$ turn respectively, by this transformation of coordinates, into the hyperplanes  $\zeta=1$ and $\zeta=\zeta_P$. Thus, integration by slices gives
$$
V_{n+1}(\mathcal{S})=\sqrt{q}\int_{\mathcal{S}^\ast}\,d{\boldsymbol{y}}\,d\zeta
=\sqrt{q}\int_1^{\zeta_P}
V_n^*(C(\zeta))\,d\zeta.
$$

The $(\boldsymbol{y},\zeta)$-equation of the hyperboloid $\mathcal{H}$ is
\begin{equation}
\label{ec:quadratic5}
\sum_{i=1}^n(1-\overline{p_i}^2)y_i^2-2\sum_{i<j}\overline{p_i}\,\overline{p_j}y_iy_j=\zeta^2-1.
\end{equation}
The section with each fixed hyperplane $\zeta=\overline{\zeta}\ge 1$ is a $n$-ellipsoid of semiaxes $\sqrt{(\overline{\zeta}^2-1)/\lambda_i}$. Here, $\lambda_i$ ($1=1,\ldots,n$) are the eigenvalues of the symmetric matrix associated to the definite positive quadratic form on the left hand side of \eqref{ec:quadratic5}. We have (see Note \ref{app2} in the Appendix)
$$
\lambda_1\cdots\lambda_n=
\begin{vmatrix}
1-\overline{p_1}^2&-\overline{p_1}\,\overline{p_2} &\cdots&-\overline{p_1}\,\overline{p_n}\\
-\overline{p_1}\,\overline{p_2} & 1-\overline{p_2}^2&\cdots&-\overline{p_2}\,\overline{p_n}\\
\vdots&\vdots&\ddots&\vdots\\
-\overline{p_1}\,\overline{p_n}&-\overline{p_2}\,\overline{p_n}&\cdots&1-\overline{p_n}^2
\end{vmatrix}
=K_{n+1}(\overline{p_1},\ldots,\overline{p_n})=q,
$$
so that, for $1\le \zeta\le \zeta_P$,
$$
V_n^*(C(\zeta))=\frac{\alpha_n(1)}{\sqrt{q}}\,(\zeta^2-1)^{n/2},
$$
where $\alpha_n$ is given in \eqref{eq:volumen}, and
$$
V_{n+1}(\mathcal{S})=\alpha_n(1)
\int_1^{\zeta_P}(\zeta^2-1)^{n/2}\,d\zeta,
$$
 as stated.

The Archimedean cone $\mathcal{K}$ has base $C(\zeta_P)$ and vertex $R$. As
\[
\frac{V_n^*(K(\zeta))}{V_n^*(C(\zeta_P))}=\Bigl(\frac{\zeta-1}{\zeta_P-1}\Bigr)^n,
\]
we have
\begin{align*}
V_{n+1}(\mathcal{K})&=\sqrt{q}\int_{\mathcal{K}^\ast}\,d{\boldsymbol{y}}\,d\zeta
=\sqrt{q}\int_1^{\zeta_P} V_n^*(K(\zeta))\,d\zeta\\
&=\sqrt{q}\frac{V_n^*(C(\zeta_P))}{(\zeta_P-1)^n}\int_1^{\zeta_P} (\zeta-1)^n\,d\zeta\\
&=\frac{\alpha_n(1)}{n+1}\cdot(\zeta_P^2-1)^{n/2} (\zeta_P-1).
\end{align*}
Therefore,
\begin{equation*}
\frac{V_{n+1}(\mathcal{S})}{V_{n+1}(\mathcal{K})}%
=\frac{n+1}{(\zeta_P^2-1)^{n/2}(\zeta_P-1)}\int_1^{\zeta_P}(\zeta^2-1)^{n/2}\,d\zeta. \qed
\end{equation*}
\renewcommand{\qed}{}
\end{proof}

\begin{rem}
\textit{a)}
Observe that, for example, for $n=1$ we have
$$
V_2(\mathcal{S})=k\sqrt{k^2-1}-\cosh^{-1}k
$$
and, for $n=2$,
$$
V_3(\mathcal{S})=\frac\pi3(k-1)^2(k+2).
$$

\textit{b)} When $n=2$ we have
\begin{equation*}
V_3(\mathcal{S})=\frac{\zeta_P+2}{\zeta_P+1}\cdot V_3(\mathcal{K}),
\end{equation*}
as Archimedes said. Observe that $\abs{RT}/\abs{OR}=\zeta_P-1$, and thus
$$
\frac{\abs{RT}+3\abs{OR}}{\abs{RT}+2\abs{OR}}=\frac{\zeta_P+2}{\zeta_P+1}.
$$
\end{rem}

The direct part of the following result for $n=2$ should be a corollary of some Archimedean proposition which should be related to Proposition \ref{prop:25-26}, analogously as Proposition \ref{prop:23} (on segments of a paraboloid) is related to Proposition \ref{prop:21-22}. But there is not in (CS) such a proposition.\footnote{Perhaps the major complication of the formula explains the gap we observe in this place of (CS) (and so later, with the segments of an ellipsoid).}

\begin{teor}\label{teor3} Let $n\ge1$. Let
$a_1,\ldots,a_{n+1}>0$ and $k>1$ be fixed, and $h(\boldsymbol{x},z)=-\sum_{i=1}^n\frac{x_i^2}{a_i^2}+\frac{z^2}{a_{n+1}^2}$. Every hyperplane tangent to any sheet of the hyperboloid
$h(\boldsymbol{x},z)=k^2$ cuts off from the corresponding sheet of the homothetic hyperboloid $h(\boldsymbol{x},z)=1$, a compact $(n+1)$-dimensional set $\mathcal{S}$  of constant Euclidean volume $G(k)=a_1\cdots
a_{n+1} G_0(k)$.

Conversely, let $\mathcal{M}\colon z=f(\boldsymbol{x})$ be a $C^{(2)}(\mathbb{R}^{n+1})$-surface lying inside the upper sheet of the hyperboloid $h(\boldsymbol{x},z)=1$.
If the hyperplane tangent to $\mathcal{M}$  at a generic point $P=(\overline{\boldsymbol{x}},\overline{z})$ ($\overline{z}=f(\overline{\boldsymbol{x}})$; $f_{x_i}(\overline{\boldsymbol{x}})=\overline{p_i}$ for $i=1,\ldots,n$)
verifies the conditions
\begin{equation}
\label{eq:condmulconverse}
 a_{n+1}^2-\sum_{i=1}^n a_i^2\overline{p_i}^2>0, \quad\text{and}\quad
\overline{z}-\sum_{i=1}^n\overline{p_i}\,\overline{x_i}\ge\sqrt{a_{n+1}^2-\sum_{i=1}^n a_i^2\overline{p_i}^2}
\end{equation}
and
cuts off from the upper sheet of $h(\boldsymbol{x},z)=1$
a compact set of constant (independent of $P$) $(n+1)$-dimensional Euclidean volume $V$,  then $\mathcal{M}$ is the upper sheet of the hyperboloid $h(\boldsymbol{x},z)=k^2$, where $k=G^{-1}(V)$ ($k>1$).
\end{teor}

\begin{proof}
Again, as in the proof of Theorem \ref{theor:par}, by using a linear transformation, it is enough to consider $a_1=\cdots=a_n=a_{n+1}=1$.

Let $\Upsilon$ be the hyperplane tangent to the upper sheet of  $h(\boldsymbol{x},z)=k^2$ at the generic point $T=(\overline{\boldsymbol{x}},\overline{z})$. We have $\overline{z}=\sqrt{k^2+\sum_{i=1}^n \overline{x_i}^2}$ and $\overline{p_i}=\overline{x_i}/\overline{z}$. The coordinates of the point $R$ are $\frac1k(\overline{\boldsymbol{x}},\overline{z})$. On the other hand, $\zeta_T=k$ and $T$ belongs to the ray $OR$. 
It suffices to apply Lemma \ref{lema3} to conclude that
$V_{n+1}(\mathcal{S})=G_0(k)$.

Conversely, we assume that the hyperplane $\Upsilon$ tangent to $\mathcal{M}$  at the point $P=(\overline{\boldsymbol{x}},\overline{z})$
cuts off from the upper sheet of $h(\boldsymbol{x},z)=1$ a compact set, and that this compact set has a constant $(n+1)$-dimensional volume  $V$. Let $k=G_0^{-1}(V)$.

With $f_{x_i}(\overline{\boldsymbol{x}})=\overline{p_i}$ for $i=1,\ldots,n$, the equation of $\Upsilon$ is
\[
z=\sum_{i=1}^n\overline{p_i}x_i+\overline{z}-\sum_{i=1}^n\overline{p_i}\,\overline{x_i}.
\]
By applying Lemma \ref{lema3} it results that
$$
V=G_0\left(\frac{\overline{z}-\sum_{i=1}^n\overline{p_i}\,\overline{x_i}}{\sqrt{1-\sum_{1=1}^n\overline{p_i}^2}}\right).
$$
But we had also
$V=G_0(k)$, and  $G_0$ is a one-to-one function on the interval $(1,\infty)$ (because
$H_0'(t)=C_n t(t^2-1)^{n/2-1}\ne 0$ in $(1,\infty)$, where $C_n$ is a constant independent of $k$). From this, we conclude that
\begin{equation}  \label{ec:hiperclaigen}
\overline{z}-\sum_{i=1}^n\overline{p_i}\,\overline{x_i}=k\,\sqrt{1-\sum_{1=1}^n\overline{p_i}^2},
\end{equation}
once again a Clairaut' first order partial differential equation, since $\overline{\boldsymbol{x}}$ runs in $\mathbb{R}^n$. The solutions of \eqref{ec:hiperclaigen} are the hyperplanes of the $n$-parametric family
\[
\phi(\boldsymbol{x},z,\lambda_1,\ldots,\lambda_n):=z-\sum_{i=1}^n\lambda_i\,x_i-k\,\sqrt{1-\sum_{1=1}^n\lambda_i^2}=0,
\]
the envelopes of $r$-parametric subfamilies with $1\le r\le n-1$, and the singular integral, the envelope of the complete family. The only acceptable solution is the latter, just the upper sheet of the two-sheeted hyperboloid $z^2-\sum_{i=1}^n x_i^2=k^2$, centrally homothetic to $z^2-\sum_{i=1}^n x_i^2=1$. Indeed, $\frac{\partial\phi}{\partial\lambda_i}=0$ gives $x_i=\frac{k\lambda_i}{\sqrt{1-\sum_{1=1}^n\lambda_i^2}}$ for $i=1,\ldots,n$, from where
$$
k^2+\sum_{i=1}^n x_i^2=\frac{k^2}{1-\sum_{i=1}^n\lambda_i^2}.
$$
Then $x_i=\lambda_i\,\sqrt{k^2+\sum_{i=1}^n x_i^2}$ for $i=1,\ldots,n$, and $z=\sqrt{k^2+\sum_{i=1}^n x_i^2}$.
\end{proof}
\subsection{The general one-sheeted hyperbolic case}

In \cite[p. 345]{poldude} it is proved that the surface of flotation of a certain portion of the solid bounded by the one-sheeted hyperboloid $z^2=x^2+y^2-1$ is a corresponding portion of a two-sheeted hyperboloid centrally homothetic of the two-sheeted \emph{complementary} hyperboloid $z^2=x^2+y^2+1$. Our next target is the proof of a $(n+1)$-dimensional analogous of this result. As usual, we start with a lemma about the volume limited by the corresponding quadric and, in this case, two hyperplanes.

\begin{lema} \label{lema3bis}
Let $n\ge 1$ and $\varepsilon>0$. Let $(\overline{p_1},\ldots,\overline{p_n})\in\mathbb{R}^n$ and $P=(\overline{\boldsymbol{x}},\overline{z})\in\mathbb{R}^{n+1}$  be fixed.  Let $d= \overline{z}-\sum_{i=1}^n\overline{p_i}\,\overline{x_i}$ and $q=1-\sum_{1=1}^n\overline{p_i}^2$.
Suppose that\footnote{With these conditions we can see that the set $\mathcal{S}_\varepsilon$, for all $\varepsilon$, is a compact set of nonempty interior which contains the origin and contains completely the \lq\lq{}base\rq\rq{} $\{(\boldsymbol{x},z)\vert z=-1/\varepsilon;\ \sum_{i=1}^n x_i^2\le 1+z^2\}$ of the truncated hyperboloid $\mathcal{H}_\varepsilon$.} $q>0$, $d\ge 0$, and $1+\varepsilon d\ge \sqrt{1+\varepsilon^2}\sqrt{1-q}$. 
The Euclidean volume of the compact set $\mathcal{S}_\varepsilon\subset E^{n+1}$ cut off by the hyperplane
$
\Upsilon\colon z-\overline{z}=\sum_{i=1}^n\overline{p_i}(x_i-\overline{x_i})
$
from the set $\mathcal{H}_\varepsilon=\{(\boldsymbol{x},z)\vert z\ge -1/\varepsilon;\  \sum_{i=1}^n x_i^2\le 1+z^2\}$ is equal to
$H_0(\zeta_P)$, where $\zeta_P=d/\sqrt{q}$ and
\begin{equation} \label{hache0}
H_0(t)=\frac{\pi^{n/2}}{\Gamma(n/2+1)}\int_{-1/\varepsilon}^t (1+\zeta^2)^{n/2}\,d\zeta.
\end{equation}

\end{lema}

\begin{proof}
Let
\[
\Upsilon'\colon z=\sum_{i=1}^n\overline{p_i}x_i
\]
be the hyperplane parallel to  $\Upsilon$ passing through the origin. By symmetry, the volume of the compact set $\mathcal{B}_0$ which $\Upsilon'$ cuts off from $\mathcal{H}_\varepsilon$ is equal to the volume of the compact set which cuts off from $\mathcal{H}_\varepsilon$ the hyperplane $z=0$, namely,
$$
V_0=V_{n+1}(\mathcal{B}_0)=\alpha_n(1)
\int_{-1/\varepsilon}^0(1+\zeta^2)^{n/2}\,d\zeta,
$$
as is easily seen. The volume of the compact set which the hyperplane $\Upsilon$ cuts off from $\mathcal{H}_\varepsilon$ is the sum of $V_0$ and the volume $V$ of the compact set $\mathcal{B}_\varepsilon\subset\mathcal{H}_\varepsilon $ enclosed between the hyperplanes $\Upsilon'$ and $\Upsilon$ (see Figure \ref{fig:onesheet}). We can compute this volume $V$ with the same change of variables \eqref{eq10} we used above in Lemma \ref{lema3}. The new $\boldsymbol{y}$-hyperplane is just $\Upsilon'$,
and the $\zeta$-axis goes through the centers of the $n$-dimensional ellipsoidal sections of $\mathcal{H}_\varepsilon$ parallel to $\Upsilon$.

\begin{figure}[h]
\begin{center}
\includegraphics[scale=.8]{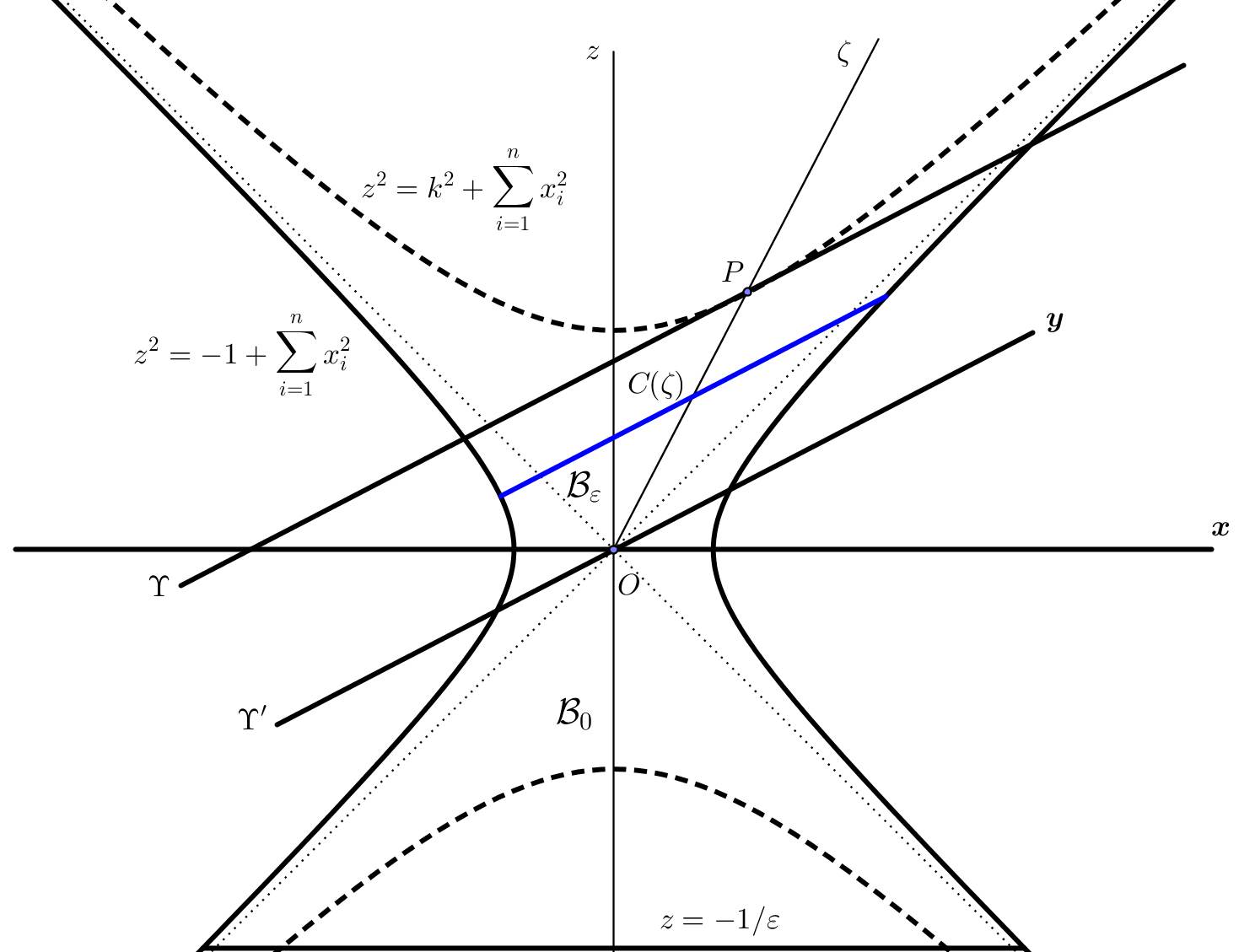}
\caption{}\label{fig:onesheet}
\end{center}
\end{figure}

The hyperplanes  $\Upsilon'$ and $\Upsilon$ turn, respectively, into the hyperplanes  $\zeta=0$ and $\zeta=\zeta_P$.  The volume of $\mathcal{B}_\varepsilon$ is then
$$
V=\sqrt{q}\int_{\mathcal{S}^\ast}\,d{\boldsymbol{y}}\,d\zeta
=\sqrt{q}\int_0^{\zeta_P}
V_n^*(C(\zeta))\,d\zeta.
$$

The $(\boldsymbol{y},\zeta)$-equation of the one-sheeted hyperboloid  is
\begin{equation}
\label{ec:quadratic9}
\sum_{i=1}^n(1-\overline{p_i}^2)y_i^2-2\sum_{i<j}\overline{p_i}\,\overline{p_j}y_iy_j=1+\zeta^2.
\end{equation}
The section with each fixed hyperplane $\zeta=\overline{\zeta}\ge 0$ is a $n$-ellipsoid of semiaxes $\sqrt{(1+\overline{\zeta}^2)/\lambda_i}$, where $\lambda_i$ ($1=1,\ldots,n$) are the eigenvalues of the symmetric matrix associated to the definite positive quadratic form on the left hand side of \eqref{ec:quadratic9}. As above in Lemma \ref{lema3} we also have, for $0\le \zeta\le \zeta_P$,
$$
V_n^*(C(\zeta))=\frac{\alpha_n(1)}{\sqrt{q}}\,(1+\zeta^2)^{n/2},
$$
so that
$V=\alpha_n(1)
\int_0^{\zeta_P}(1+\zeta^2)^{n/2}\,d\zeta$,
and
$$
V_{n+1}(\mathcal{S}_\varepsilon)=V_0+V=\alpha_n(1)\int_{-1/\varepsilon}^{\zeta_P}(1+\zeta^2)^{n/2}\,d\zeta,
$$
as stated.
\end{proof}

With the help of the previous lemma, proceeding as in Theorem \ref{theor:par} and Theorem \ref{teor3} (leaving the details for the reader), we have the following result.

\begin{teor} Let $n\ge1$,
$a_1$, \ldots, $a_{n+1}>0$, $k>0$ and $\varepsilon>0$ be fixed, and $h(\boldsymbol{x},z)=-\sum_{i=1}^n\frac{x_i^2}{a_i^2}+\frac{z^2}{a_{n+1}^2}$. Every hyperplane tangent to the upper sheet of the two-sheeted hyperboloid
$h(\boldsymbol{x},z)=k^2$ in a point $T=(\overline{\boldsymbol{x}},\overline{z})$ such that $\overline{z}\le \varphi(\varepsilon)$, where
\begin{equation} \label{condrecond}
\varphi(\varepsilon):=\frac k\varepsilon\Bigl(k+\sqrt{(1+\varepsilon^2)(1+k^2)}\Bigr)
\end{equation}
cuts off, from the set $\mathcal{H}_\varepsilon=\{(\boldsymbol{x},z)\vert z\ge -1/\varepsilon;\ h(\boldsymbol{x},z)\ge -1\}$, a $(n+1)$-dimensional compact set $\mathcal{S}_\varepsilon$  of constant Euclidean volume $H(k)=a_1\cdots a_{n+1} H_0(k)$.

Conversely, let $\varepsilon>0$ be fixed and $\mathcal{M}\colon z=f(\boldsymbol{x})$ be a $C^{(2)}$-surface lying inside the upper cone $\mathcal{C}^+=\{(\boldsymbol{x},z)\vert z>0,\ h(\boldsymbol{x},z)\ge 0\}$.
If the hyperplane tangent to $\mathcal{M}$ at any generic point $P=(\overline{\boldsymbol{x}},\overline{z})$ ($\overline{z}=f(\overline{\boldsymbol{x}})$) cuts off from the set $\mathcal{H}_\varepsilon=\{(\boldsymbol{x},z)\vert z\ge -1/\varepsilon;\ h(\boldsymbol{x},z)\ge -1\}$  a $(n+1)$-dimensional compact set which contains completely the base\footnote{Let $f_{x_i}(\overline{\boldsymbol{x}})=\overline{p_i}$ for $i=1,\ldots,n$, and $d= \overline{z}-\sum_{i=1}^n\overline{p_i}\,\overline{x_i}$.  They must satisfy the conditions
\begin{equation*}
 a_{n+1}^2-\sum_{i=1}^n a_i^2\overline{p_i}^2>0, \quad
d\ge 0, \quad \text{and}\quad (1+\varepsilon d)a_{n+1}\ge
\sqrt{1+\varepsilon^2 a_{n+1}^2}\sqrt{\sum_{i=1}^n a_i^2\overline{p_i}^2}.
\end{equation*}
} of $\mathcal{H}_\varepsilon$ and has a constant (independent of $P$) volume $V$,  then $\mathcal{M}$ forms the upper sheet of the two-sheeted hyperboloid $h(\boldsymbol{x},z)=k^2$, where $k=H^{-1}(V)$ ($k>0$), truncated according to the condition $z\le\varphi(\varepsilon)$.
\end{teor}

\subsection{The case of the general cone}


With the change of variables \eqref{eq10} used in Lemma \ref{lema3} and Lemma \ref{lema3bis}, we can prove the following lemma giving the volume limited by a cone and a hyperplane. The proof is similar to the given one for Lemma \ref{lema3} so we omit the details.

\begin{lema} \label{lema3ter}
Let $n\ge 1$, $(\overline{p_1},\ldots,\overline{p_n})\in\mathbb{R}^n$ and $P=(\overline{\boldsymbol{x}},\overline{z})\in\mathbb{R}^{n+1}$  be fixed.  Let $d= \overline{z}-\sum_{i=1}^n\overline{p_i}\,\overline{x_i}$ and $q=1-\sum_{i=1}^n \overline{p_i}^2$.
Suppose that $q>0$ and $d\ge 0$. The Euclidean volume of the cone $\mathcal{S}\subset E^{n+1}$ cut off by the hyperplane
$\Upsilon\colon z-\overline{z}=\sum_{i=1}^n\overline{p_i}(x_i-\overline{x_i})$
from the set $\mathcal{C}=\{(\boldsymbol{x},z)\vert z\ge 0; \sum_{i=1}^n x_i^2\le z^2\}$ is equal to
$C_0(\zeta_P)$, where $\zeta_P=d/\sqrt{q}$ and
\begin{equation} \label{ce0}
C_0(t)=\frac{\pi^{n/2}}{\Gamma(n/2+1)}\;\frac{t^{n+1}}{n+1}.
\end{equation}
\end{lema}

We find references concerning the floating body for the cone in the case $n=2$, for example, in \cite[p. 31-32]{bravais}, \cite[p. 49-50]{besant}, and \cite[p. 346]{poldude}. The corresponding result for the $(n+1)$-dimensional case is the following one. We omit the proof, due to its similarity to the proofs of Theorem \ref{theor:par} and Theorem \ref{teor3}. A proof of the next result, from a geometric point of view, can be found in \cite[Theorem 3.4]{kimandsong}.

\begin{teor}\label{teor5}  Let $n\ge1$,
$a_1$, \ldots, $a_n>0$ and $k>0$  be fixed, and $h(\boldsymbol{x},z)=z^2-\sum_{i=1}^n\frac{x_i^2}{a_i^2}$. Every hyperplane tangent to the upper sheet of the two-sheeted hyperboloid
$h(\boldsymbol{x},z)=k^2$ cuts off from the cone $\mathcal{C}=\{(\boldsymbol{x},z)\vert z\ge 0;\ h(\boldsymbol{x},z)\ge 0\}$, a $(n+1)$-dimensional compact set $\mathcal{S}$  of constant Euclidean volume $C(k)=a_1\cdots a_n C_0(k)$.

Conversely, the $C^{(2)}$-surface which envelopes the family of the hyperplanes cutting off from the cone $\mathcal{C}$ a compact set\footnote{Let $P=(\overline{\boldsymbol{x}},\overline{z})$ be a generic point and $(\overline{p_1},\ldots,\overline{p_n},-1)$ be the normal vector of any hyperplane of that family. The conditions $1-\sum_{i=1}^n a_i^2\overline{p_i}^2>0$ and $\overline{z}-\sum_{i=1}^n\overline{p_i}\,\overline{x_i}\ge 0$ must be fulfilled.} of constant volume $V$ is the upper sheet of the two-sheeted hyperboloid $h(\boldsymbol{x},z)=k^2$ where $k=C^{-1}(V)$ ($k>0$).
\end{teor}

\section{The general elliptic case}

Concerning the volumes of the segments cut off by a plane from an ellipsoid obtained by revolving an ellipse around its major or minor axis (\emph{spheroids}), Archimedes proves in (CS) three pairs of propositions, according to the volume of the segment be equal to, or less than, or greater than, half the volume of the ellipsoid (see \cite[p. 141-150: (CS), Propositions 27-32]{heath}, and also \cite[p. 259]{dijks}). Comparing the volume $V(\mathcal{S})$ of the ellipsoidal segment to the volume $V(\mathcal{K})$ of the cone with the same base and axis (i.e., the segment joining the center of the base and the center of the ellipsoid), Archimedes derives, for example in the second case (see Figure~\ref{cir2} below for the used notation),
\[
\frac{V(\mathcal{S})}{V(\mathcal{K})}
=\frac{3\abs{OR}-\abs{RT}}{2\abs{OR}-\abs{RT}}.
\]

\begin{lema} \label{lema4}
Let $n\ge 1$ and $0<k<1$. Let $(\overline{p_1},\ldots,\overline{p_n})\in\mathbb{R}^n$ and $(\overline{\boldsymbol{x}},\overline{z})\in\mathbb{R}^{n+1}$  be fixed, verifying $0<\overline{z}^2+\sum_{i=1}^n\overline{x_i}^2<1$.

The Euclidean volume of the smaller cup $\mathcal{S}\subset E^{n+1}$ cut off from the sphere $z^2=1-\sum_{i=1}^n x_i^2$ by the hyperplane
$
\Upsilon\colon z-\overline{z}=\sum_{i=1}^n\overline{p_i}(x_i-\overline{x_i}),
$
is equal to $L_0(d)$, where
$$
d=\frac{\abs{\overline{z}-\sum_{i=1}^n\overline{p_i}\,\overline{x_i}}}{\sqrt{1+\sum_{1=1}^n\overline{p_i}^2}}
$$
is the distance from the origin to the hyperplane $\Upsilon$, and
\begin{equation} \label{ele0}
L_0(t)=\frac{\pi^{n/2}}{\Gamma(n/2+1)}\int_t^1 (1-\zeta^2)^{n/2}\,d\zeta.
\end{equation}

In fact, this volume is equal to $\frac{n+1}{(1-d^2)^{n/2}(1-d)}\int_d^1(1-\zeta)^{n/2}\,d\zeta$ times the volume of the cone which has the same base and the same axis as the cup.
\end{lema}

\begin{proof}
Integration by spherical slices $C(t)$ of measures $\alpha_n(\sqrt{1-t^2})$  (see Figure~\ref{cir2}) yields, for the $(n+1)$-dimensional
Euclidean volume of  $\mathcal{S}$,  the expression

\begin{figure}[h]
\begin{center}
\includegraphics[scale=0.8]{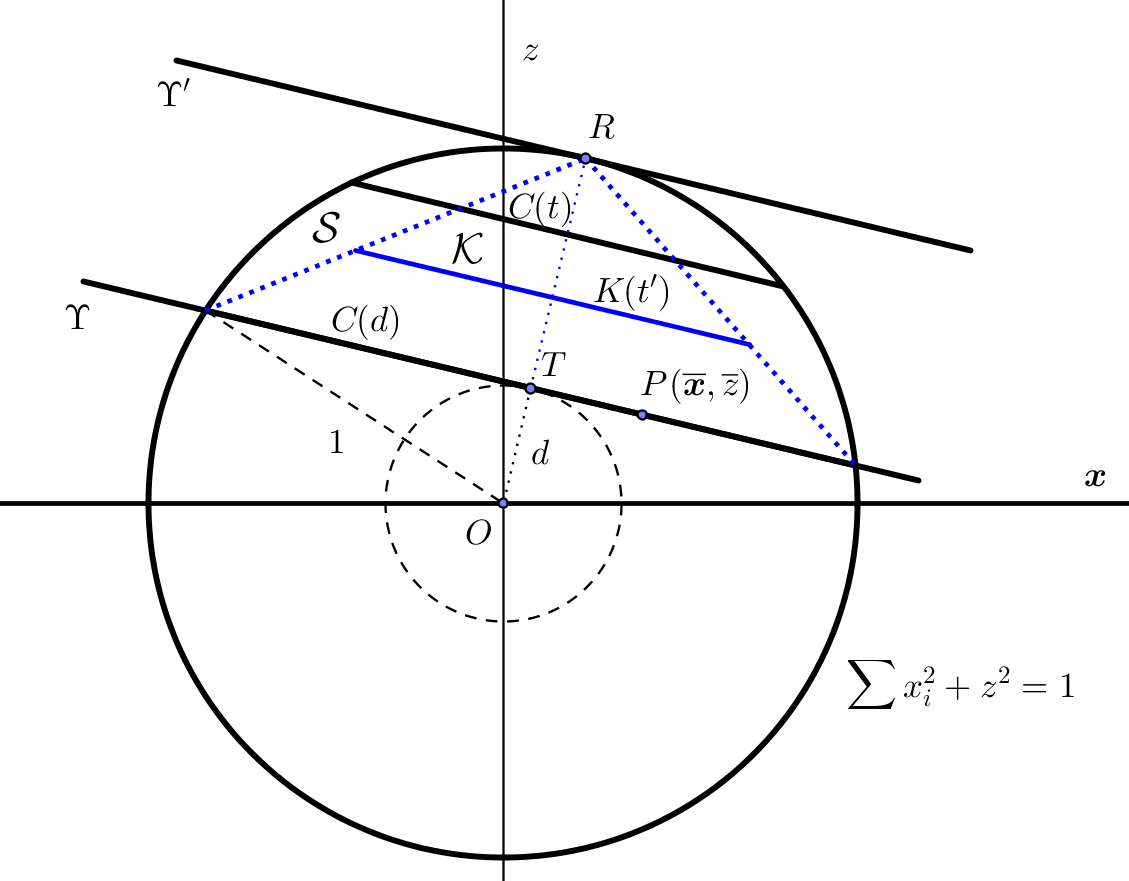}
\caption{} \label{cir2}
\end{center}
\end{figure}

\begin{equation*}
V_{n+1}(\mathcal{S})=\int_{d}^1 V_n(C(t))\,dt=\frac{\pi^{n/2}}{\Gamma(n/2+1)}\int_{d}^1 (1-t^2)^{n/2}\,dt.
\end{equation*}

The Archimedean cone $\mathcal{K}$ has base $C(d)$ and vertex $R$. As
\[
\frac{V_n(K(t))}{V_n(C(d))}=\Bigl(\frac{1-t}{1-d}\Bigr)^n,
\]
we have
\begin{align*}
V_{n+1}(\mathcal{K})&=\int_d^1 V_n(K(t))\,dt=\frac{V_n(C(d))}{(1-d)^n}\int_d^1 (1-t)^n\,dt\\
&=\frac{\alpha_n(1)}{n+1}\cdot(1-d^2)^{n/2} (1-d).
\end{align*}

Therefore,
\begin{equation*}
\frac{V_{n+1}(\mathcal{S})}{V_{n+1}(\mathcal{K})}%
=\frac{n+1}{(1-d^2)^{n/2}(1-d)}\int_d^1(1-t)^{n/2}\,dt. \qed
\end{equation*}
\renewcommand{\qed}{}
\end{proof}

\begin{rem}
\textit{a)} Observe that, for example, for $n=1$ we have
$$
V_2(\mathcal{S})=\cos^{-1}k-k\sqrt{1-k^2}
$$
and, for $n=2$,
$$
V_3(\mathcal{S})=\frac\pi3(1-k)^2(k+2).
$$

\textit{b)} When $n=2$, since $\abs{RT}/\abs{OR}=1-d$, we have, as Archimedes said,
\begin{equation*}
\frac{V_3(\mathcal{S})}{V_3(\mathcal{K})}=\frac{2+d}{1+d}=\frac{3\abs{OR}-\abs{RT}}{2\abs{OR}-\abs{RT}}.
\end{equation*}
\end{rem}

Our last theorem, which can be proved analogously as Theorem \ref{theor:par} and Theorem \ref{teor3}, characterizes the surfaces of flotation for ellipsoids.

\begin{teor}
\label{thm:ellipsoid}
Let $n\ge 1$ and $0<k<1$. Let
$a_1,\ldots,a_{n+1}>0$ be fixed, and $\ell(\boldsymbol{x},z)=\sum_{i=1}^n\frac{x_i^2}{a_i^2}+\frac{z^2}{a_{n+1}^2}$. The volume of the compact set ($(n+1)$-dimensional) $\mathcal{S}$ of smaller Euclidean volume  cut off from the ellipsoid $\ell(\boldsymbol{x},z)=1$ by any hyperplane tangent to the ellipsoid
$\ell(\boldsymbol{x},z)=k^2$, is constant and equal to $L(k)=a_1\cdots
a_{n+1} L_0(k)$.

Conversely, let $\mathcal{M}=\phi^{-1}(0)$ be a $C^{(2)}$-ovaloid  lying inside the ellipsoid $\ell(\boldsymbol{x},z)=1$.
If the  volume of the compact set ($(n+1)$-dimensional) $\mathcal{S}$ of smaller Euclidean volume cut off from the ellipsoid $\ell(\boldsymbol{x},z)=1$ by any hyperplane tangent to $\mathcal{M}$ is constant and equal to $V$,
then (except for a multiplicative constant) $\phi(\boldsymbol{x},z)=\ell(\boldsymbol{x},z)-k^2$, where $k=L^{-1}(V)$.
\end{teor}

\begin{rem}
In Theorem \ref{theor:par} we characterized the surface of flotation for a paraboloid. To conclude our paper, we observe that Theorem \ref{theor:par} can be deduced in a nice way from Theorem \ref{thm:ellipsoid}, as can be found in \cite[p. 50-51]{besant} for $n=2$ (see also \cite[p. 345]{poldude}).

Indeed, from Theorem \ref{thm:ellipsoid}, the surface of flotation of an ellipsoidal floating body of semiaxes $a_1$, \ldots, $a_n$, and $a_{n+1}$, that is, the envelope of the hyperplanes cutting off from this ellipsoid a compact segment of constant volume $V$, is the concentric and homothetic ellipsoid of semiaxes $ka_1$, \ldots $ka_n$, and $ka_{n+1}$, where $k=L^{-1}(V)\in (0,1)$, so that
$$
V=a_1\cdots
a_{n+1}\frac{\pi^{n/2}}{\Gamma(n/2+1)}\int_k^1 (1-\zeta^2)^{n/2}\,d\zeta.
$$

If the equation of the paraboloid floating is $2z=\sum_{i=1}^n\frac{x_i^2}{\alpha_i}$, then, from the floating ellipsoid $\sum_{i=1}^n\frac{x_i^2}{a_i^2}+\frac{(z-a_{n+1})^2}{a_{n+1}^2}=1$, let us make $a_1$, \ldots, $a_n$, and $a_{n+1}$ tend to infinity in such a way that $\frac{a_i^2}{a_{n+1}}\to\alpha_i$ for each $i=1,\ldots, n$. If $V$ denotes the finite volume immersed, then $\frac{V}{a_1\cdots a_na_{n+1}}\to 0$, so that $k\to 1$. Hence the surface of flotation is an identical paraboloid. The distance between its vertex and the vertex of the given paraboloid  is the limiting value of $a_{n+1}(1-k)$, namely,
$$
\gamma=\left(\frac{\Gamma(n/2+1)V}{(2\pi)^{n/2}\sqrt{\alpha_1\cdots\alpha_n}}\right)^{\frac2{n+2}}.
$$
 Therefore, the surface of flotation is the paraboloid $2(z-\gamma)=\sum_{i=1}^n\frac{x_i^2}{\alpha_i}$.
\end{rem}

\section*{Appendix}

\begin{append}\label{app1}
Let $n\ge 1$.
$$
J_{n+1}(a_1,\ldots,a_n)=\begin{vmatrix}
1&0&0&\cdots&0&a_1\\
0&1&0&\cdots&0&a_2\\
0&0&1&\cdots&0&a_3\\
\vdots&\vdots&\vdots&\ddots&\vdots&\vdots\\
0&0&0&\cdots&1&a_n\\
a_1&a_2&a_3&\cdots&a_n&1
\end{vmatrix}=1-\sum_{i=1}^{n}a_i^2.
$$
\end{append}

\begin{proof}
By induction. We have $J_2(a)=1-a^2$ and, assuming that $J_{n}(a_1,\ldots,a_{n-1})=1-\sum_{i=1}^{n-1}a_i^2$, we get, developing by the first row,

\begin{align*}
J_{n+1}(a_1,\ldots,a_n)&=J_{n}(a_2,\ldots,a_n)+(-1)^{n+2}a_1\cdot
\begin{vmatrix}
0&1&0&\cdots&0\\
0&0&1&\cdots&0\\
\vdots&\vdots&\vdots&\ddots&\vdots\\
0&0&0&\cdots&1\\
a_1&a_2&a_3&\cdots&a_n
\end{vmatrix} \\
&=1-\sum_{i=2}^{n}a_i^2+(-1)^{n+2}(-1)^{n+1}a_1^2=1-\sum_{i=1}^{n}a_i^2,
\end{align*}
as required.
\end{proof}

\begin{append} \label{app2}   Let $n\ge 1$.
$$
K_n(a_1,\ldots,a_n)=\begin{vmatrix}
1-a_1^2&-a_1a_2&\cdots& -a_1a_{n-1}& -a_1a_n\\
-a_1a_2&1-a_2^2&\cdots& -a_2a_{n-1}&-a_2a_n\\
\vdots&\vdots&\ddots&\vdots&\vdots\\
-a_1a_{n-1}&-a_2a_{n-1}&\cdots&1-a_{n-1}^2 &-a_{n-1}a_n\\
-a_1a_n&-a_2a_n&\cdots&-a_{n-1}a_n&1-a_n^2
\end{vmatrix}=1-\sum_{i=1}^{n}a_i^2.
$$
\end{append}

\begin{proof}
By induction.
For $n=1$ is trivial. For $n=2$,
\begin{align*}
K_2(a_1,a_2)&=\begin{vmatrix}
1-a_1^2&-a_1a_2\\
-a_1a_2&1-a_2^2
\end{vmatrix}
=1-(a_1^2 +a_2^2).
\end{align*}
Assume the result for all positive integers smaller than $n$. We can develop $K_n$ by the Desnanot-Jacobi adjoint matrix theorem \cite[Theorem 3.12, p. 111-113]{lcarroll}:
\begin{multline}\label{enapp3}
K_n(a_1,\ldots,a_n)=\frac1{K_{n-2}(a_2,\ldots,a_{n-1})}\\
\cdot\left\{ K_{n-1}(a_1,\ldots,a_{n-1})\cdot K_{n-1}(a_2,\ldots,a_n)
-D_{n-1}^2\right\},
\end{multline}
where we have denoted by $D_{n-1}$ (and we have used that the determinant of a matrix is equal to the determinant of the transposed) the determinant
\begin{align*}
D_{n-1}&=
\begin{vmatrix}
-a_1a_2&-a_1a_3& \cdots &-a_1a_{n-1}&-a_1a_n\\
1-a_2^2&-a_2a_3&\cdots&-a_2a_{n-1}& -a_2a_n\\
-a_3a_2&1-a_3^2&\cdots&-a_3a_{n-1}& -a_3a_n\\
\vdots&\vdots&\ddots&\vdots&\vdots\\
-a_{n-1}a_2&-a_{n-1}a_3&\cdots&1-a_{n-1}^2
&-a_{n-1}a_n
\end{vmatrix}
\\
%
&=(-1)^{n-1}a_2a_3\cdots a_n
\begin{vmatrix}
a_1&a_1& \cdots &a_1&a_1\\
-\frac1{a_2}+a_2&a_2&\cdots&a_2&a_2\\
a_3&-\frac1{a_3}+a_3&\cdots&a_3&a_3\\
\vdots&\vdots&\ddots&\vdots&\vdots\\
a_{n-1}&a_{n-1}&\cdots&-\frac1{a_{n-1}}+a_{n-1}
&a_{n-1}
\end{vmatrix}
\\
\phantom{D_{n-1}}
&=(-1)^{n-1}(-1)^{n-2}a_2a_3\cdots a_n
\begin{vmatrix}
0&0& \cdots &0&a_1\\
\frac1{a_2}&0&\cdots&0& a_2\\
0&\frac1{a_3}&\cdots&0& a_3\\
\vdots&\vdots&\ddots&\vdots&\vdots\\
0&0&\cdots&\frac1{a_{n-1}} &a_{n-1}
\end{vmatrix} \\
&=-a_2a_3\cdots a_n\cdot (-1)^n \,
\frac{a_1}{a_2a_3\cdots a_{n-1}}=(-1)^{n-1}a_1a_n.
\end{align*}

Substituting in \eqref{enapp3} and applying the induction hypothesis  it turns out
\begin{align*}
K_n(a_1,\ldots,a_n)&=\frac1{1-\sum_{i=2}^{n-1}a_i^2}\cdot
\left(
(1-\sum_{i=1}^{n-1}a_i^2)\cdot
(1-\sum_{i=2}^{n}a_i^2)
-a_1^2a_n^2\right)\\
&=1-\sum_{i=1}^{n}a_i^2,
\end{align*}
as required.
\end{proof}

\section*{Acknowledgements}
The authors thank their colleagues Pilar Benito and Luis J. Hernández for fruitful discussions which improved the presentation of the paper. They are also grateful to the referee for the thorough reading of the manuscript and for providing reference \cite{kimandsong}.



\end{document}